\documentclass[11pt]{amsart}
\usepackage{amsmath}
\usepackage{amssymb}
\usepackage{amsfonts}
\usepackage{amsthm}
\usepackage{enumerate}
\usepackage[mathscr]{eucal}
\usepackage{graphicx}
\usepackage[bookmarksnumbered, bookmarksopen, colorlinks, citecolor=blue, linkcolor=blue]{hyperref}

\usepackage{eufrak}
\usepackage{color}
\usepackage{tikz-cd}
\usepackage{enumerate}
\usepackage{cleveref}

\usetikzlibrary{decorations.pathreplacing}

\newtheorem{theorem}{Theorem}[section]
\newtheorem{lemma}[theorem]{Lemma}

\theoremstyle{definition}
\newtheorem{definition}[theorem]{Definition}

\newtheorem{proposition}[theorem]{Proposition}

\theoremstyle{remark}
\newtheorem{remark}[theorem]{Remark}

\numberwithin{equation}{section}

\def\intslash{\rlap{\kern  .32em $\mspace {.5mu}\backslash$ }\int}
\def\qsl{{\rlap{\kern  .32em $\mspace {.5mu}\backslash$ }\int_{Q_x}}}

\def\set#1{{ \left\{ #1 \right\} }}

\def\dist{{\text{\it dist}}}

\newcommand{\Be}{\begin{equation}}
\newcommand{\Ee}{\end{equation}}
\newcommand{\Bes}{\begin{equation*}}
\newcommand{\Ees}{\end{equation*}}
\newcommand{\Bsp}{\begin{split}}
\newcommand{\Esp}{\end{split}}
\newcommand{\Bm}{\begin{multline}}
\newcommand{\Em}{\end{multline}}
\newcommand{\Bea}{\begin{eqnarray}}
\newcommand{\Eea}{\end{eqnarray}}
\newcommand{\Beas}{\begin{eqnarray*}}
\newcommand{\Eeas}{\end{eqnarray*}}
\newcommand{\Benu}{\begin{enumerate}}
\newcommand{\Eenu}{\end{enumerate}}
\newcommand{\Bi}{\begin{itemize}}
\newcommand{\Ei}{\end{itemize}}

\newcommand{\supp}{\operatorname{supp}}

\newcommand{\mQ}{\operatorname{\mathcal{Q}}}

\newcommand{\mF}{\operatorname{\mathcal{F}}}
\newcommand{\mB}{\operatorname{\mathcal{B}}}
\newcommand{\mD}{\operatorname{\mathcal{D}}}
\newcommand{\mH}{\operatorname{\mathcal{H}}}
\newcommand{\mL}{\operatorname{\mathcal{L}}}
\newcommand{\mT}{\operatorname{\mathcal{T}}}

\newcommand{\R}{\operatorname{\mathbb{R}}}
\newcommand{\C}{\operatorname{\mathbb{C}}}
\newcommand{\Z}{\operatorname{\mathbb{Z}}}

\newcommand{\sS}{\operatorname{\mathbb{S}}}

\newcommand*{\dif}{{\rm d}}

\def\set #1{ \left\{ #1 \right\} }

\usepackage{hyperref}

\begin{document}

\title[Weak (1,1) estimates for singular integral operators]{weak type (1,1) bound criterion for maximal truncated singular integral operators with rough kernels}

\author{Xudong Lai}
\address{Xudong Lai: Institute for Advanced Study in Mathematics\\
Harbin Institute of Technology\\
Harbin
150001\\
China;
Zhengzhou Research Institute\\
Harbin Institute of Technology\\
Zhengzhou
450000\\
China}
\email{xudonglai@hit.edu.cn}

\author{Lingxin Xu}
\address{Lingxin Xu: Institute for Advanced Study in Mathematics, Harbin Institute of Technology, Harbin, 150001, China}
\email{xulingxin97@163.com}




\keywords{singular integral operators, rough kernels, weak (1,1) estimates}

\begin{abstract}
This paper proves the weak type (1,1) boundedness for maximal truncated singular integral operators with rough non-convolution kernels of the form $\Omega(x-y) K(x,y)$, where $\Omega$ is homogeneous of degree zero with $\Omega \in L\log L(\sS^{d-1})$ and $K$ is a standard Calder\'on-Zygmund kernel.
This extends the recent resolution of the classical convolution case by Lai \cite{Lai2025weak} to a wider class of operators.

%

\end{abstract}

\maketitle

\section{Introduction and state of the main result}
Let $d\geq 2$, and let $\Omega$ be a measurable function on $\R^d\setminus\{0\}$ whose restriction to $\sS^{d-1}$ belongs to $L^1(\sS^{d-1})$.
We assume that $\Omega$ is homogeneous of degree zero:
    \begin{equation}\label{homogeneous function of degree zero}
    \Omega(r\theta)=\Omega(\theta),
    \qquad \text{for all}\quad \theta\in\sS^{d-1},\quad r>0,
    \end{equation}
and satisfies the cancellation condition
    \begin{equation}\label{cancellation condition}
    \int_{\sS^{d-1}}\Omega(\theta)\, \dif \theta=0.
    \end{equation}
The associated homogeneous singular integral operator is
    \[
    T_{\Omega}f(x)
    =\mathrm{p.v.}\int_{\R^d}
    \frac{\Omega(x-y)}{|x-y|^d}f(y)\, \dif y.
    \]
The angular function $\Omega$ is not assumed to be smooth, so the
operator is called a singular integral with rough kernel.

In 1956, Calder\'on and Zygmund \cite{Calderon1956Onsingular}
proved that $T_{\Omega}$ is bounded on $L^p(\R^d)$ for $1<p<\infty$
when $\Omega$ is odd and belongs to $L^1(\sS^{d-1})$, or when
$\Omega$ is even and belongs to $L\log L(\sS^{d-1})$.
Strong $L^1$ boundedness generally fails for singular integral operators.
The Hilbert transform is a familiar example. The natural endpoint
question is therefore whether $T_{\Omega}$ is of weak type $(1,1)$.
In 1988, Christ and Rubio de Francia \cite{Christ1988Weaktype}
proved this estimate in dimension two under the assumption
$\Omega\in L\log L(\sS^1)$.
Independently, Hofmann \cite{Hofmann1988Weak} obtained the estimate
under the stronger assumption $\Omega\in L^q(\sS^1)$ for some $q>1$.
In 1996, Seeger \cite{Seeger1996Singular} established the estimate
in every dimension $d\geq 2$ for $\Omega\in L\log L(\sS^{d-1})$.
For related results, see
\cite{Christ1988Weak,Christ1988Theweak,Seeger2014Aweak,Sjogren1997Rough,Tao1999Theweak}.

A harder endpoint question concerns the maximal truncated operator
    \[
    \mT_{\Omega,*}f(x)
    =\sup_{\epsilon>0}  \Big|
    \int_{|x-y|>\epsilon}
    \frac{\Omega(x-y)}{|x-y|^d}f(y)\, \dif y  \Big|.
    \]
This operator requires control of all truncation radii at once.
The lack of angular regularity makes such control difficult.
Calder\'on and Zygmund \cite{Calderon1956Onsingular} also established
its $L^p$ boundedness for $1<p<\infty$ under the odd and even kernel
assumptions stated above.
Its weak type $(1,1)$ boundedness remained open for many years,
even for bounded angular functions.
Seeger \cite{Seeger2014Aweak} highlighted this question in 2014.
In a 2025 preprint, Lai \cite{Lai2025weak} proved the weak type $(1,1)$
estimate for $\Omega\in L\log L(\sS^{d-1})$ satisfying \eqref{homogeneous function of degree zero} and \eqref{cancellation condition}.


In this paper, we study the more general maximal truncated operator:
    \begin{align} \label{maximal truncated operator}
	T_{\Omega,*} f(x) 
	= \sup_{\epsilon>0} \Big| \int_{|x-y|>\epsilon} \Omega(x-y) K(x,y) f(y)\dif y\Big|, 
    \end{align}
where $K:\R^d \times \R^d \setminus \set{(x,x): x\in\R^d} \rightarrow \C$ is a binary function.
We assume the size condition
    \begin{equation}\label{Size condition}
    |K(x,y)|\leq \frac{C_1}{|x-y|^d},\qquad x\neq y,
    \end{equation}
and the following regularity estimates for some $0<\gamma\leq 1$
and constants $C_2,C_3>0$:
    \begin{equation}\label{regularity condition}
    \begin{aligned}
    |K(x_1,y)-K(x_2,y)|
    &\leq C_2\frac{|x_1-x_2|^\gamma}{|x_1-y|^{d+\gamma}},
    && |x_1-y|\geq 2|x_1-x_2|,\\
    |K(x,y_1)-K(x,y_2)|
    &\leq C_3\frac{|y_1-y_2|^\gamma}{|x-y_1|^{d+\gamma}},
    && |x-y_1|\geq 2|y_1-y_2|.
   \end{aligned}
   \end{equation}
We assume that $\Omega\in L\log L(\sS^{d-1})$ and satisfies
\eqref{homogeneous function of degree zero}.
Taking $K(x,y)=|x-y|^{-d}$ recovers $\mT_{\Omega,*}$.

Our aim is to obtain a weak type $(1,1)$ estimate for this more general
class under an $L^2$ boundedness assumption on the maximal operator
For $\Omega\not\equiv 0$, set
\begin{align} \label{C_Omega}
 	C_{\Omega}
 	:= \|\Omega\|_{L\log L(\sS^{d-1})}	 + \int_{\sS^{d-1}} |\Omega(\theta)| \Big( 1+ \log^+\big( |\Omega(\theta)| / \|\Omega\|_{L^1(\sS^{d-1})} \big) \Big) \dif\theta.
 	\end{align}
where $\log^+t=\max\{\log t,0\}$ for $t>0$.
The integrand is defined to be zero where $\Omega=0$.
When $\Omega\equiv 0$, we set $C_{\Omega}=0$.
Then, the main result is stated as follows:
 
    \begin{theorem} \label{main theorem}
    Suppose $d\geq 2$ and let $T_{\Omega,*}$ be defined as in (\ref{maximal truncated operator}), where \(K\) satisfies \eqref{Size condition} and \eqref{regularity condition}, and \(\Omega\in L\log L(S^{d-1})\) satisfies \eqref{homogeneous function of degree zero}. 
    Assume further that \(T_{\Omega,*}\) is bounded on \(L^2(\R^d)\). 
 	Then, $T_{\Omega,*}$ is of weak type $(1,1)$.
 	More precisely, for any $f\in L^1(\R^d)$ and $\lambda>0$, 
 	\[
 	\lambda \big| \{x\in\R^d: T_{\Omega,*}f(x)>\lambda \} \big| \lesssim C_{\Omega} \|f\|_{L^1(\R^d)},
 	\]
 	where the constant $C_{\Omega}$ is defined as in (\ref{C_Omega}).   
 	\end{theorem}

We now describe three examples of the maximal truncated operator \eqref{maximal truncated operator}.
For each application, one must verify the kernel estimates and the
$L^2$ boundedness of the maximal operator.
The detailed kernel estimates and the proofs of these applications are given in Section~\ref{Section: applications of the main theorem}.

\medskip
\noindent\textbf{Maximal Calder\'on commutator.}
For a Lipschitz function $A$ on $\R^d$, the Calder\'on commutator
\cite{Calderon1965Commutators} is
    \begin{align*} 
	T_{\Omega,A} f(x) 
	= \text{p.v. } \int_{\R^d} \frac{\Omega(x-y)}{|x-y|^d} \cdot \frac{A(x)-A(y)}{|x-y|} \cdot f(y)\dif y.
    \end{align*}
Its maximal truncation is
    \begin{align} \label{maximal Calderon commutator}
	T_{\Omega,A,*} f(x) 
	= \sup_{\epsilon>0} \Big| \int_{|x-y|>\epsilon} \frac{\Omega(x-y)}{|x-y|^d} \cdot \frac{A(x)-A(y)}{|x-y|} \cdot f(y)\dif y\Big|,
    \end{align}
where $\Omega$ satisfies (\ref{homogeneous function of degree zero}), and the following cancellation condition of order 1,
\begin{equation}\label{cancellation condition of order 1}
    \int_{\sS^{d-1}}\Omega(\theta)\theta^\alpha\, \dif \theta=0,
    \quad \text{for all}\ \alpha \in \Z_+^d 
    \ \text{with}\ |\alpha|=1.
\end{equation}
Under this condition and $\Omega\in L\log L(\sS^{d-1})$, the theorem of Baj\v{s}anski and Coifman
\cite{Bajvsanski1967On}, gives
the $L^p$ boundedness of the maximal commutator for $1<p<\infty$.
Here
    \[
    K_A(x,y)=\frac{A(x)-A(y)}{|x-y|^{d+1}}
    \]
satisfies \eqref{Size condition} and \eqref{regularity condition}
with $\gamma=1$ and constants controlled by the Lipschitz seminorm of $A$.
Therefore, Theorem~\ref{main theorem} gives the weak type $(1,1)$ estimate for
$T_{\Omega,A,*}$.

\medskip
\noindent\textbf{Maximal Calder\'on commutator of Baj\v{s}anski--Coifman type.}
Let $l\geq 1$ be an integer.
Assume that $A\in C^{l-1}(\R^d)$ and that its derivatives of order $l-1$
are globally Lipschitz. Thus its weak derivatives of order $l$ are bounded.
Write $A_\alpha=\partial^\alpha A$ and define
    \begin{align} \label{P_l(A,x,y)}
    P_l(A,x,y)
    =A(x)-\sum_{|\alpha|<l}\frac{A_\alpha(y)}{\alpha!}(x-y)^\alpha.
    \end{align}
The corresponding singular integral is
    \begin{align*} 
    T_{\Omega,A,l} f(x) = \mathrm{p.v.} \int_{\mathbb{R}^d} \frac{\Omega(x - y)}{|x - y|^d} \cdot \frac{P_l(A, x, y)}{|x - y|^l} \cdot f(y) \dif y,
    \end{align*}
and its maximal truncation is
   \begin{align} \label{maximal Calderon commutator of B-C type}
    T_{\Omega,A,l,*} f(x) 
	= \sup_{\epsilon>0} \Big| \int_{|x-y|>\epsilon} \frac{\Omega(x-y)}{|x-y|^d} \cdot \frac{P_l(A,x,y)}{|x-y|^l} \cdot f(y)\dif y\Big|.
    \end{align}
When $l=1$, these operators agree with the Calder\'on commutator
and its maximal truncation, respectively.
We assume $\Omega$ satisfies (\ref{homogeneous function of degree zero}) and the cancellation condition of order $l$,
    \begin{align}\label{cancellation condition of order l}
    \int_{\sS^{d-1}} \Omega(\theta) \theta^{\alpha} \dif\theta = 0,
    \quad \text{for all}\ \alpha \in \Z_+^d 
    \ \text{with}\ |\alpha|=l,
    \end{align} 
The theorem of Baj\v{s}anski and Coifman \cite{Bajvsanski1967On}, gives
the $L^p$ boundedness of $T_{\Omega,A,l,*}$ for $1<p<\infty$
under this condition and $\Omega\in L\log L(\sS^{d-1})$.
The kernel
    \[
    K_{A,l}(x,y)=\frac{P_l(A,x,y)}{|x-y|^{d+l}}
    \]
satisfies \eqref{Size condition} and \eqref{regularity condition}
with $\gamma=1$.
These assumptions allow us to apply Theorem~\ref{main theorem} to $T_{\Omega,A,l,*}$.

\medskip
\noindent\textbf{Maximal singular integral of Muckenhoupt type.}
Muckenhoupt \cite{Muckenhoupt1960On} studied singular integrals with an imaginary radial exponent.
For $r\in\R\setminus\{0\}$, define
    \begin{equation*} 
    T_{\Omega,ir} f(x) = \mathrm{p. v.} \int_{\mathbb{R}^d} \frac{\Omega(x-y)}{|x-y|^{d+ir}} \cdot f(y) \dif y.
    \end{equation*}
For $\Omega\in L^1(\sS^{d-1})$ satisfying
\eqref{homogeneous function of degree zero} and \eqref{cancellation condition}.
The maximal truncation is
    \begin{align} \label{maximal general singular integral of Muckenhoupt type}
	T_{\Omega,ir,*} f(x) = \sup_{\epsilon>0} \Big| \int_{|x-y|>\epsilon} \frac{\Omega(x-y)}{|x-y|^{d+ir}} \cdot f(y) \dif y \Big|.
    \end{align}
It is bounded on $L^p(\R^d)$ for $1<p<\infty$ by
\cite[Theorem~8]{Muckenhoupt1960On}.
The kernel 
    \[
    K_r(x,y)= \frac{1}{ |x-y|^{d+ir} }
    \]
satisfies
\eqref{Size condition} and \eqref{regularity condition} with $\gamma=1$.
Thus, if $\Omega\in L\log L(\sS^{d-1})$,
Theorem~\ref{main theorem} gives a weak type $(1,1)$ estimate for
$T_{\Omega,ir,*}$ for each fixed $r\neq 0$.

\medskip
This paper is organized as follows.
In Section~\ref{Section: Proof of the main theorem: reductions}, we introduce the maximal dyadic truncated operator, apply the Calder\'on-Zygmund decomposition, smooth the kernel, and perform a further dyadic decomposition, thereby reducing Theorem~\ref{main theorem} to the core square function estimate for the bad function (Lemma~\ref{lemma T_{J,2}^s}).
Section~\ref{Section: proof of lemma T_{J,2}^s} proves Lemma~\ref{lemma T_{J,2}^s} via a geometric decomposition of dyadic cubes into nested layers. 
This reduces the problem to the layerwise estimate in Lemma~\ref{lemma (T_{J_2}^s b_{j-s}) chi_{J^m}}.
In Section~\ref{Section: Proof of Lemma {lemma (T_{J_2}^s b_{j-s}) chi_{J^m}}}, we prove this layerwise estimate. 
We first linearize the maximal operator via the Rademacher-Menshov theorem, reducing it to Lemma~\ref{epsilon_u beta_u}. The proof of Lemma~\ref{epsilon_u beta_u} is then carried out using smoothing and microlocal decomposition, which further reduces the problem to four technical estimates \eqref{estimate for I_1}-\eqref{estimate for I_4}. These estimates are proved in Section~\ref{Section: Proof of estimates I_1-I_4}.
Section~\ref{Section: applications of the main theorem} discusses
the applications described above.
Finally, two classical results - the H\"ormander-Mihlin multiplier theorem and the Rademacher-Menshov theorem - are stated in Appendix \ref{Appendix} for the reader's convenience.

We end this section with notation used throughout the paper.
\begin{itemize}
 \item $\mD_j$ denotes the collection of standard dyadic cubes of side length $2^j$.
 \item $\mD=\bigcup_{j\in\Z}\mD_j$ denotes the standard dyadic grid.
 \item $Q_{x,n}$ is the unique cube in $\mD_n$ containing $x\in\R^d$.
 \item $c(Q)$ and $\ell(Q)$ denote the center and side length of a cube $Q$.
 \item For $t>0$, $tQ$ is the cube with center $c(Q)$ and side length $t\ell(Q)$.
 \item $B_r(x)$ is the open ball of radius $r>0$ centered at $x$; we write $B_r=B_r(0)$.
 \item $|E|$ is the Lebesgue measure of a measurable set $E\subset\R^d$.
 \item $\#E$ is the cardinality of a finite set $E$.
 \item $\lfloor a\rfloor$ and $\lceil a\rceil$ denote the floor and ceiling of $a$.
 \item The Fourier transform and its inverse are
 \[
  \mathcal Ff(\xi)=\int_{\R^d}e^{-i\langle x,\xi\rangle}f(x)\,dx,
  \qquad
  \mathcal F^{-1}g(x)=(2\pi)^{-d}\int_{\R^d}e^{i\langle x,\xi\rangle}g(\xi)\,d\xi.
 \]
 \item $\|\Omega\|_\infty=\|\Omega\|_{L^\infty(\sS^{d-1})}$ and
 $\|\Omega\|_1=\|\Omega\|_{L^1(\sS^{d-1})}$.
\end{itemize}

\medskip
\section{Proof of the Theorem \ref{main theorem}: reductions} \label{Section: Proof of the main theorem: reductions}

This section reduces Theorem~\ref{main theorem} to a core square function estimate.  
We first introduce a maximal dyadic truncated operator $T_*$ and show that the weak type $(1,1)$ bound for $T_{\Omega,*}$ follows from the same bound for $T_*$. 
Then we apply the Calder\'on-Zygmund decomposition to split the function into a good function (easy to handle) and a bad function.  
For the bad function, we smooth the operator and perform a further dyadic decomposition, which finally leads to a core square function estimate.  
Once this estimate is established (in the next section), the proof of the main theorem will be complete.

\medskip
\subsection{Reduction to the dyadic maximal operator}

We first introduce a maximal dyadic truncated operator $T_*$ and show that it suffices to prove the weak type $(1,1)$ bound for $T_*$.

Let $\phi$ be a nonnegative radial smooth function supported in $\{x\in\R^d:2^{-4}<|x|<2^{-2}\}$ such that $\sum\limits_{j\geq 1} \phi(2^{-j}x) =1$, for $x\in \R^d\backslash\{0\}$.
Set $\phi_j(x)= \phi(2^{-j}x)$. 
Define the dyadic kernel
    \begin{align*}
    K_j(x,y) = \phi_j(x-y)K(x,y), \quad j\in\Z.
    \end{align*}
The corresponding dyadic operator is given by
    \begin{align*}
	T_j f(x) = \int_{\R^d} \Omega(x-y) K_j(x,y) f(y) \dif y.
    \end{align*}    
We also introduce two maximal operators:
    \[
    M_\Omega f(x)=\sup_{r>0} r^{-d}\int_{|x-y|<r}|\Omega(x-y)f(y)|\,dy,
    \]
    \[
    T_* f(x)=\sup_{k\in\mathbb{Z}}\Bigl|\sum_{j\ge k} T_j f(x)\Bigr|.
    \]
A direct estimate gives the pointwise control:
    \begin{align*} 
    T_*f(x)\lesssim M_\Omega f(x) + T_{\Omega,*}f(x), 
    \quad
    T_{\Omega,*}f(x)\lesssim M_\Omega f(x)+T_*f(x),
    \end{align*}   
which allows us to reduce the study of $T_{\Omega,*}$ to that of $T_*$.

Since \(\Omega\in L^1(\mathbb S^{d-1})\), the operator \(M_\Omega\) is bounded on \(L^2(\mathbb R^d)\). By the \(L^2\) boundedness assumption on \(T_{\Omega,*}\) in Theorem~\ref{main theorem}, the first estimate implies that \(T_*\) is also bounded on \(L^2(\mathbb R^d)\).
Moreover, \(M_\Omega\) is of weak type \((1,1)\) when \(\Omega\in L\log L(\mathbb S^{d-1})\) (see \cite{Christ1988Weaktype}). The second estimate therefore reduces the proof of Theorem~\ref{main theorem} to a weak type \((1,1)\) estimate for \(T_*\).

Thus, to prove Theorem \ref{main theorem}, it suffices to establish the following proposition. 
    \begin{proposition} \label{estimate for T*}
	The maximal operator $T_{*}$ is of weak type $(1,1)$, that is, for any $f\in L^1(\R^d)$ and $\lambda>0$, 
 	\begin{align*} 
     \lambda \big| \{x\in\R^d: T_*f(x)>\lambda \} \big| \lesssim C_{\Omega} \|f\|_{L^1(\R^d)}.	
     \end{align*}
 	where $C_{\Omega}$ is a finite constant defined as in (\ref{C_Omega}).
    \end{proposition}

\medskip
\subsection{Calder\'on-Zygmund decomposition}
To handle \(T_*\), we apply the Calder\'on-Zygmund decomposition to the function \(f\).

Fix $f\in L^1(\R^d)$ and $\lambda>0$. 
We perform a Calder\'on-Zygmund decomposition of $f$ at height $\lambda/C_{\Omega}$ with respect to the dyadic grid $\mD$.
This yields a collection of dyadic cubes $\mQ \subseteq \mD$ and a decomposition $ f = g + b $ satisfying the following properties:
    \begin{flalign*}
    &\text{(cz-i)} && \quad \|g\|_{L^{\infty}(\R^d)} \lesssim \lambda / C_{\Omega}, \quad \|g\|_{L^1(\mathbb{R}^d)} \lesssim \|f\|_{L^1(\mathbb{R}^d)}, \\
    &\text{(cz-ii)} && b = \sum_{Q \in \mathcal{Q}} b_Q,\quad \int_Q b_Q(x) \, dx = 0,\quad \operatorname{supp}(b_Q) \subseteq Q; &\\
    &\text{(cz-iii)} && \|b_Q\|_{L^1(\mathbb{R}^d)} \lesssim \langle f \rangle_Q |Q| \quad \text{for all}\ Q \in \mathcal{Q}, \quad \text{where}\ \langle f \rangle_Q \approx \lambda / C_{\Omega}; &\\
    &\text{(cz-iv)} && \text{All dyadic cubes in } \mathcal{Q} \text{ are disjoint}; &\\
    &\text{(cz-v)} && \Big| \bigcup_{Q \in \mathcal{Q}} Q \Big| \lesssim \frac{C_{\Omega}}{\lambda} \|f\|_{L^1(\mathbb{R}^d)}. &
    \end{flalign*}
From $f=g+b$, we obtain
    \begin{align*}
    \big| \{x\in\R^d: T_*f(x)>\lambda \} \big| 
 	\leq \big| \{x\in\R^d: T_* g(x)>\lambda/2 \} \big| 
 	+ \big| \{x\in\R^d: T_* b(x)>\lambda/2 \} \big|, 
    \end{align*}
so proving Proposition \ref{estimate for T*} splits into estimating the good function and the bad function.

For the good function $g$, using the Chebyshev inequality, the $L^2$ boundedness of $T_*$ (with norm $C \|\Omega\|_{L\log L(\sS^{d-1})}$), and (cz-i), we get
    \begin{align*}
	\big| \{x\in\R^d: T_* g(x)>\lambda/2 \} \big| 
	&\lesssim \lambda^{-2} \|T_* g\|_{L^2(\R^d)}^2
	\lesssim \lambda^{-2} \|\Omega\|_{L\log L(\sS^{d-1})}^2 \|g\|_{L^2(\R^d)}^2  \\
	&\lesssim \lambda^{-1} C_{\Omega} \|f\|_{L^1(\R^d)}.
    \end{align*}
Thus, it remains to consider the bad function.

Set $E = \bigcup\limits_{Q \in \mathcal{Q}} Q$ and  $E^* = \bigcup\limits_{Q \in \mathcal{Q}} 2^{300} Q$ (the expanded union).
We first reduce to the estimate outside $E^*$.
Notice that
    \begin{align*}
	|\{x\in\R^d: T_* b(x)>\lambda/2 \}| 
	\lesssim |E^*| 
	+ \big|\{x\in (E^*)^c: T_* b(x)>\lambda/2 \} \big|. 
    \end{align*}
From (cz-iv) and (cz-v), we have $|E^*| \lesssim |E| \lesssim \lambda^{-1} C_{\Omega} \|f\|_{L^1(\R^d)}$.
Therefore, to establish the estimate for the bad function, we only need to prove
    \begin{align} \label{estimate for T_* b}
	\big| \{x\in (E^*)^c: T_* b(x)>\lambda/2 \} \big| 
	\lesssim \lambda^{-1}C_{\Omega} \|f\|_{L^1(\R^d)},
    \end{align}
which requires a more involved process.
We first make some preparations in the next subsection, and then continue to work with it.

\medskip
\subsection{Preparations for estimating the bad function}
To prove \eqref{estimate for T_* b}, we perform two preparatory steps: first, a smoothing convolution, and then a further dyadic decomposition that localizes the operators.

\medskip
\noindent\textbf{Step 1: Smoothing the dyadic kernel.}
Let $\eta\in C_c^{\infty}$ be a nonnegative, radial function with $\int_{\R^d} \eta =1$, $\supp\eta \subset \{x\in\R^d: |x|<1 \}$, and $\sum\limits_{j\geq 1} 2^{-jd} \eta(2^{-j} x) =1$.
Set $\eta_j := 2^{-jd} \eta(2^{-j} x)$.
For a large integer $s\ge200$, define
    \begin{align*}
    s_{\gamma} = \lfloor 2\gamma^{-1}\log_2 s\rfloor+10,
    \end{align*}
where $0<\gamma<1$ is a small constant to be chosen later.
Immediately, we see
    \begin{align} \label{2^{s_{gamma}}}
    2^{s_{\gamma}} \approx s^{2 \gamma^{-1}}.
    \end{align}
Now, define the smoothed dyadic kernel
    \begin{align*} 
    K_j^s(x,y) = \int_{\R^d} \int_{\R^d} \eta_{j-s_{\gamma}}(x-u) \eta_{j-s_{\gamma}}(y-v) K_j(u,v) \dif u \dif v,
    \end{align*}
and the corresponding smoothed operator   
    \begin{align} \label{smoothed operator T_j^s}
    T_j^s h(x) = \int_{\R^d} \Omega(x-y) K_j^s(x,y) h(y) \dif y.
    \end{align}

A useful estimate for $K_j^s$ is
    \begin{align} \label{estimate for K_j^s(x,y)}
    |K_j^s(x,y)| \lesssim 2^{-jd} \chi_{ \{2^{j-5}\leq |x-y| \leq 2^{j-1}\} }.
    \end{align}
Indeed, if $K_j(u,v)\neq0$, then the support of $\phi_j$ gives $2^{j-4}<|u-v|<2^{j-2}$. 
If additionally $\eta_{j-s_\gamma}(x-u)\neq0$ and $\eta_{j-s_\gamma}(y-v)\neq0$, then $|x-u|,|y-v|<2^{j-s_\gamma}\le2^{j-10}$. 
Consequently,
    \begin{align*}
    &|x-y| \leq |u-v| + |x-u| + |v-y| \leq 2^{j-1}, \\
    &|x-y| \geq |u-v| - |x-u| - |v-y| \geq 2^{j-5},
    \end{align*}
so $(x,y)$ lies in the stated annulus. 
Moreover, by the size condition (\ref{Size condition}) and the support of $\phi_j$, 
    \begin{align} \label{|K_j(u,v)|}
    |K_j(u,v)| \leq |u-v|^{-d} \chi_{ \{2^{j-4} < |u-v| < 2^{j-2} \} }
               \lesssim 2^{-jd}.
    \end{align} 
Using the above estimate and integrating over $u,v$ gives (\ref{estimate for K_j^s(x,y)}).

\medskip
\noindent\textbf{Step 2: Further dyadic decomposition.}
For each $j\in\Z$, let $\mD_j$ be the collection of dyadic cubes of side length $2^j$:
    \begin{align*}
    \mD_j = \big\{2^j\cdot [m_1,(m_1+1))\times \cdots \times [m_d,(m_d+1)): \ (m_1, \cdots, m_d)\in\Z^d \big\},
    \end{align*}
and set
    \begin{align*}
    \mD = \bigcup_{j\in\Z} \mD_j.
    \end{align*}
Fix a $j\in\Z$.
For a dyadic cube $J\in\mD_j$, define
    \begin{align*}
    T_J^s h(x)
    = \int_{\R^d} \Omega(x-y) K_j^s(x,y) h(y) \chi_J(y) \dif y.
    \end{align*}
Using $\sum\limits_{J\in\mD_j} \chi_J = 1$, we obtain the dyadic decomposition of the smoothed operator $T_j^s$:
    \begin{align} \label{T_j^s = sum T_J^s}
    T_j^s h = \sum_{J\in\mD_j} T_J^s h.
    \end{align}
From the support property of $K_j^s$ (see (\ref{|K_j(u,v)|})) and the fact that $y\in J$, one easily deduces that 
    \begin{align} \label{supp T_J^s h}
    \supp \big(T_J^s h \big) \subseteq 3J,
    \end{align}
i.e., the support of $T_J^s h$ is contained in the cube concentric with $J$ but three times larger.

With the two steps above, we are ready to estimate \eqref{estimate for T_* b}.

\medskip
\subsection{Estimate for the bad function I: Reduction to the smoothed operator}
Applying the smoothed operator \(T_j^s\) constructed as in (\ref{smoothed operator T_j^s}), we can split \(T_*b\) into a non-smooth part and a smooth part.
We then prove the estimate for the non-smooth part; the smooth part will be handled in the next subsection.

First, for each $s\in\mathbb{Z}$, let $\mathcal{Q}_s=\{Q\in\mathcal{Q}:\ell(Q)=2^s\}$ and define $b_s = \sum\limits_{Q\in\mQ_s} b_Q$. 
Applying (cz-iv), (cz-iii), (cz-v), we obtain
    \begin{align}\label{sum_s ||b_s||_1}
    \sum_{s\in\Z}  \|b_s\|_{L^1(\R^d)} 
    = \sum_{s\in\Z} \sum_{Q\in\mQ} \|b_Q\|_{L^1(\R^d)}
    \lesssim \lambda C_{\Omega}^{-1} \sum_{s\in\Z} \sum_{Q\in\mQ} |Q|
    \lesssim \|f\|_{L^1(\R^d)}.
    \end{align}  
Otherwise, a key observation is the following vanishing property:
    \begin{align} \label{T_j b_{j-s} (x) = 0}
	T_j b_{j-s} (x) = 0, \quad \text{for}\ x\in (E^*)^c, \ s<200.
    \end{align}
Indeed, recall $T_j b_{j-s}(x) = \sum\limits_{Q\in\mQ_{j-s}} \int_{Q} K_j(x,y) b_Q(y) \dif y$.
Since \(K_j(x,y)=0\) unless \(2^{j-4}<|x-y|<2^{j-2}\), it suffices to show that for every \(Q\in\mathcal{Q}_{j-s}\) and every \(y\in Q\), we have \(|x-y|\ge 2^{j-2}\) for \(x\in(E^*)^c\). 
To prove this, by the definition of the distance between two sets,
    \begin{align*}
	|x-y| \geq \dist\big((E^*)^c, Q \big)
	      \geq \dist\big((2^{300}Q)^c, Q \big)
	      = \big((2^{300}-1)/2 \big) \ell(Q)
	      \geq 2^{298+j-s}
	      >2^{j-2},
    \end{align*}
where the second inequality follows from $E^* \supset 2^{300}Q$.
Hence (\ref{T_j b_{j-s} (x) = 0}) holds.

Now, using the decomposition \(b=\sum\limits_{s\in\Z}b_s\) (from (cz-ii) and the definition \(b_s=\sum\limits_{Q\in\mathcal{Q}_s}b_Q\)), the vanishing property \eqref{T_j b_{j-s} (x) = 0}, and the splitting \(T_j=(T_j-T_j^s)+T_j^s\), we obtain that for \(x\in (E^*)^c\),
    \begin{align*}
    T_* b (x)
    &= \sup_{k\in\Z} \Big|\sum_{j\geq k} \sum_{s\in\Z}T_j b_{j-s} (x) \Big|
    \leq \sum_{s\geq 200} \sup_{k\in\Z} \Big|\sum_{j\geq k} T_j b_{j-s} (x) \Big| \\
    &\leq \sum_{s\geq 200} \sup_{k\in\Z} \Big|\sum_{j\geq k} \big(T_j - T_j^s\big) b_{j-s} (x) \Big| 
    + \sum_{s\geq 200} \sup_{k\in\Z} \Big|\sum_{j\geq k} T_j^s b_{j-s} (x) \Big|.
    \end{align*}    
Therefore, (\ref{estimate for T_* b}) follows once we establish the two lemmas below.
    \begin{lemma} \label{T_j - T_j^s}
    With the above notation,
    \begin{align*}
	\Big|\{x\in (E^*)^c: \sum_{s\geq 200} \sup_{k\in\Z} \Big|\sum_{j\geq k} \big(T_j - T_j^s\big) b_{j-s} (x) \Big| >\lambda/4 \} \Big| 
	\lesssim \lambda^{-1}C_{\Omega} \|f\|_{L^1(\R^d)}.
	\end{align*}
    \end{lemma}

	\begin{lemma} \label{T_j^s}
	 With the above notation,
	\begin{align} \label{equality T_j^s}
	\Big|\{x\in (E^*)^c: \sum_{s\geq 200} \sup_{k\in\Z} \Big|\sum_{j\geq k} T_j^s b_{j-s} (x) \Big| >\lambda/4 \} \Big| 
	\lesssim \lambda^{-1}C_{\Omega} \|f\|_{L^1(\R^d)}.
	\end{align}
	\end{lemma}
For brevity, we prove only Lemma~\ref{T_j - T_j^s} here; the proof of Lemma~\ref{T_j^s} is more involved and is deferred to the next subsection.

\begin{proof}[Proof of Lemma \ref{T_j - T_j^s}]
By the definition of $K_j^s$ and $\int_{\R^d} \eta_{j-s_{\gamma}} = 1$, we have
    \begin{align*}
    \big| \big( K_j-K_j^s \big)(x,y) \big|
    &\leq \int_{\R^d} \int_{\R^d} \eta_{j-s_{\gamma}}(x-u) \eta_{j-s_{\gamma}}(y-v) \big| \big( K_j(x,y)-K_j(u,v) \big) \big| \dif u \dif v \\
    &\leq \sup_{{\begin{subarray}{c} u,v\in \\ \supp\eta_{j-s_{\gamma}} \end{subarray}}} \big| K_j(x,y) - K_j(x-u,y-v) \big|.
    \end{align*}
We split the difference $\big| K_j(x,y) - K_j(x-u,y-v) \big|$ as 
    \begin{align*}
    \big| K_j(x,y) - K_j(x,y-v) \big| + \big| K_j(x,y-v) - K_j(x-u,y-v) \big|.
    \end{align*}
It suffices to estimate the first term, since the second one can be handled similarly.

By the definition of $K_j$,
    \begin{align*}
	\big| K_j(x,y) - K_j(x,y-v) \big|
	&\leq \big| \phi_j(x-y) \big| \big| K(x,y) - K(x,y-v) \big| \\
	&\quad + \big| \phi_j(x-y) - \phi_j(x-y-v) \big| \big| K(x,y-v) \big| \\
	& =: I_1 + I_2.
    \end{align*}
We aim to show 
    \begin{align*} 
    I_i \lesssim s^{-2} 2^{-jd} \chi_{ \{2^{j-5} < |x-y| < 2^{j-1}\} },
    \end{align*}
for $i=1, 2$.

\medskip
\noindent\textbf{Estimate of \(I_1\).} 
We first verify that $|x-y|>2|v|$, which allows us to apply the regularity condition (\ref{regularity condition}).
Indeed, when $\phi_j (x-y) \neq 0$, we have $2^{j-4} < |x-y| < 2^{j-2}$. 
Meanwhile, since $v \in \supp \eta_{j-s_{\gamma}} \subset \{z:\in\R^d: |z| < 2^{j-s_{\gamma}} \}$ and $s_{\gamma} \geq 10$, we get 
    \[
    2|v| < 2^{j-s_{\gamma} + 1} \leq 2^{j-9} < |x-y|.
    \]

Thus, using the regularity condition \eqref{regularity condition} and \eqref{2^{s_{gamma}}},
    \begin{align*}
    \big| K(x,y) - K(x,y-v) \big|
    \lesssim \frac{|v|^{\gamma}}{|x-y|^{d+\gamma}}
    \lesssim 2^{- \gamma s_{\gamma}} 2^{-jd}
    \lesssim s^{-2} 2^{-jd}. 
    \end{align*}
Together with the support of $\phi_j$, we obtain
    \begin{align*}
    I_1 \lesssim s^{-2} 2^{-jd} \chi_{ \{2^{j-4} <|x-y|< 2^{j-2}\} }.	
    \end{align*}

\medskip
\noindent\textbf{Estimate of \(I_2\).}
We first determine the support of $I_2$.
When $\phi_j (x-y) - \phi_j (x-y-v) \neq 0$, one of the following cases holds:
    \[
    (i)\ 2^{j-4} < |x-y| < 2^{j-2}; \qquad
    (ii)\ 2^{j-4} < |x-y-v'| < 2^{j-2}.
    \]
If $(i)$ holds, we directly have $2^{j-5} < |x-y| < 2^{j-1}$.
In either case, using $|v| < 2^{j-s_{\gamma}} < 2^{j-10} $,  
    \begin{align*}
    2^{j-5} \leq |x-y-v| - |v| \leq |x-y| \leq |x-y-v| + |v| \leq 2^{j-1}.
    \end{align*}
This provides the desired support property.    
Moreover, a similar calculation shows that $|x-y-v| \approx 2^j$ whenever $\phi_j (x-y) - \phi_j (x-y-v) \neq 0$.
Hence, by the size condition (\ref{Size condition}), the mean value theorem, $|v| < 2^{j-s_{\gamma}}$, and (\ref{2^{s_{gamma}}}),
    \begin{align*}
    I_2 \leq |x-y-v|^{-d} \big( |v| 2^{-j} \big) \chi_{ \{2^{j-5} < |x-y| < 2^{j-1} \} } 
        \lesssim s^{-2} 2^{-jd} \chi_{ \{2^{j-5} < |x-y| < 2^{j-1} \} } .
    \end{align*}

Consequently, from the estimates for $I_1$ and $I_2$,
    \begin{align*}
    \big| \big( K_j-K_j^s \big)(x,y) \big|
    \lesssim s^{-2} 2^{-jd} \chi_{ \{2^{j-5} < |x-y| < 2^{j-1} \} },
    \end{align*}
which implies 
    \begin{align*}
    \big\| \big(T_j - T_j^s\big) b_{j-s} \big\|_{L^1(\R^d)}
    &= \int_{\R^d} \Big| \int_{\R^d} \Omega(x-y) \big( K_j-K_j^s \big)(x,y) b_{j-s}(x)  \dif y \Big| \dif x \\
    &\lesssim s^{-2} 2^{-jd} \int_{\R^d} \Big( \int_{\R^d} \big| \Omega(x-y) \big| \chi_{ \{2^{j-5} \leq|x-y|\leq 2^{j-1} \} } \dif y \Big) \big| b_{j-s}(x) \big| \dif x \\
    &\lesssim s^{-2} \|\Omega\|_{1} \|b_{j-s}\|_{L^1(\R^d)}.
    \end{align*}
Finally, by the Chebyshev inequality,  
    \begin{align*}
    \Big|\{x\in (E^*)^c: \sum_{s\geq 200}  \sup_{k\in\Z} & \Big|\sum_{j\geq k} \big(T_j - T_j^s\big) b_{j-s} (x) \Big| >\lambda/4 \} \Big| \\ 
    &\lesssim \lambda^{-1} \sum_{s\geq 200} \sum_{j\in\Z} \big\| \big(T_j - T_j^s\big) b_{j-s}  \big\|_{L^1(\R^d)} \\
    &\lesssim \lambda^{-1} \sum_{s\geq 200} s^{-2} \|\Omega\|_{1} \sum_{j\in\Z} \|b_{j-s}\|_{L^1(\R^d)},
    \end{align*}
which is bounded by $\lambda^{-1} \|\Omega\|_{1} \|f\|_{L^1(\R^d)}$ by using (\ref{sum_s ||b_s||_1}).    
This completes the proof of Lemma \ref{T_j - T_j^s}.  	
\end{proof}

It remains to prove Lemma \ref{T_j^s}.
Thus, we have reduced the estimate for the bad function to that of the smoothed operator \(T^s_j\).

\medskip
\subsection{Estimates for the bad function II: Reduction to the core square function}

In this subsection, we further decompose the operator \(T_j^s\) by using the dyadic decomposition $T_j^s=\sum\limits_{J\in\mD_j}T_J^s$ (see (\ref{T_j^s = sum T_J^s})) and then split the function \(\Omega\) according to the size of $|\Omega(\theta)|$.

To begin with, by $T_j^s=\sum\limits_{J\in\mD_j}T_J^s$, the left-hand side of (\ref{equality T_j^s}) becomes
    \begin{align} \label{rewrite the left‑hand side of ({T_j^s})}
    \Bigl|\Bigl\{x\in(E^*)^c:\sum_{s\ge200}\sup_{k\in\mathbb{Z}}\Bigl|\sum_{j\ge k}\sum_{J\in\mD_j} T_J^s b_{j-s}(x)\Bigr|>\lambda/4\Bigr\}\Bigr|.
    \end{align}
To proceed, choose a sufficiently small constant $\iota>0$ (to be fixed later) and set
    \begin{align*}
    D^{\iota,s} = \big\{ \theta\in\sS^{d-1}: |\Omega(\theta)| > 2^{\iota s} \|\Omega\|_{1} \big\}.
    \end{align*}
Define $\Omega_1 = \Omega \chi_{D^{\iota,s}}$ and $\Omega_2 = \Omega \chi_{(D^{\iota,s})^c}$.
Correspondingly, we obtain \(T_J^s = T_{J,1}^s+T_{J,2}^s\) with
    \begin{align*}
    & T_{J,1}^s h(x) = \int_{\R^d} \Omega_1(x-y) K_j^s(x,y) h(y) \chi_J(y) \dif y, \\
    & T_{J,2}^s h(x) = \int_{\R^d} \Omega_2(x-y) K_j^s(x,y) h(y) \chi_J(y) \dif y.
    \end{align*}    
Now, applying the Chebyshev inequality to \eqref{rewrite the left‑hand side of ({T_j^s})}, we get
    \begin{align*}
    (\ref{rewrite the left‑hand side of ({T_j^s})})
    &\leq \lambda^{-1} \sum_{s\geq 200} \sum_{j\in\Z} \sum_{J\in\mD_j} \big\| T_{J,1}^s b_{j-s} \big\|_{L^1(\R^d)} \\
    &\quad+ \lambda^{-2} \Big( \sum_{s\geq 200}  \Big\| \sup_{k\in\Z} \Big| \sum_{j\geq k} \sum_{J\in\mD_j} T_{J,2}^s b_{j-s} \Big| \Big\|_{L^2(\R^d)} \Big)^2.
    \end{align*}
Thus, the proof of \eqref{equality T_j^s} is reduced to the following two lemmas:   
    \begin{lemma} \label{lemma T_{J,1}^s}
    With the above notation,
    \begin{align*}
    \sum_{s\geq 200} \sum_{j\in\Z} \sum_{J\in\mD_j} \big\| T_{J,1}^s b_{j-s} \big\|_{L^1(\R^d)}
	\lesssim C_{\Omega}  \|f\|_{L^1(\R^d)}.	
    \end{align*}
    \end{lemma}

    \begin{lemma} \label{lemma T_{J,2}^s}
    With the above notation, for any $s>200$, there exists a positive constant $\delta$ and a positive integer $N_0$ such that
    \begin{align*}
    \Big\| \sup_{k\in\Z} \Big| \sum_{j\geq k} \sum_{J\in\mD_j} T_{J,2}^s b_{j-s} \Big| \Big\|_{L^2(\R^d)}
	\lesssim 2^{-\delta s} s^{ \frac{N_0}{2} \gamma^{-1}+1}  \big( \lambda C_{\Omega} \|f\|_{L^1(\R^d)} \big)^{1/2}.
    \end{align*}
    \end{lemma}

The proof of Lemma~\ref{lemma T_{J,1}^s} is straightforward by using (\ref{estimate for K_j^s(x,y)}) and the definition of $\Omega_1$; we include it here for completeness.
However, the proof of Lemma~\ref{lemma T_{J,2}^s} is considerably more involved; it requires a linearization of the maximal operator, a careful geometric partition of the dyadic cubes, and oscillatory integral estimates. 
We postpone its proof to Section~\ref{Section: proof of lemma T_{J,2}^s}.

\begin{proof}[Proof of Lemma \ref{lemma T_{J,1}^s}]
By (\ref{estimate for K_j^s(x,y)}), 
    \begin{align*}
    \big\| T_{J,1}^s b_{j-s} \big\|_{L^1(\R^d)}
    &\leq 2^{-jd} \int_{\R^d} |\Omega_1(x-y)| \chi_{ \{2^{j-5} < |x-y| < 2^{j-1}\} } \Big( \int_{\R^d} \big| b_{j-s}(y) \big| \chi_J(y) \dif y \Big) \dif x \\
    &= 2^{-jd} \Big(\int_{2^{j-5}}^{2^{j-1}} \int_{\sS^{d-1}} |\Omega_1(\theta)| \dif\theta \dif r \Big) \big\| b_{j-s}\chi_J \big\|_{L^1(\R^d)} \\
    &\lesssim \Big( \int_{\sS^{d-1}} |\Omega_1(\theta)| \dif\theta \Big) \| b_{j-s}\chi_J \|_{L^1(\R^d)}.
    \end{align*}
Then, using (\ref{sum_s ||b_s||_1}),
    \begin{align*}
    \sum_{s\geq 200} \sum_{j\in\Z} \sum_{J\in\mD_j} \big\| T_{J,1}^s b_{j-s} \big\|_{L^1(\R^d)}
    &\lesssim \sum_{j\in\Z} \| b_{j-s} \|_{L^1(\R^d)} \int_{\sS^{d-1}} |\Omega(\theta)| \sum_{s\geq 200}\chi_{D^{\iota,s}}(\theta)  \dif\theta \\
    &\lesssim \|f\|_{L^1(\R^d)} \int_{\sS^{d-1}}  |\Omega(\theta)| \Big(1+\log^+(|\Omega(\theta)\|/\|\Omega\|_{1}) \Big) \dif\theta \\
    &\lesssim C_{\Omega} \|f\|_{L^1(\R^d)}.
    \end{align*}
\end{proof}

Consequently, the proof of Lemma \ref{T_j^s} is reduced to Lemma~\ref{lemma T_{J,1}^s} (proved above) and Lemma~\ref{lemma T_{J,2}^s} (whose proof is postponed to the next section).

\medskip
\section{Proof of Lemma \ref{lemma T_{J,2}^s}: the core square function estimate}   \label{Section: proof of lemma T_{J,2}^s}
This section is devoted to the proof of Lemma~\ref{lemma T_{J,2}^s}, i.e. the $L^2$ estimate for $T_{J,2}^s$ (the part of $|\Omega(\theta)| \leq 2^{\iota s} \|\Omega\|_{L_1(\sS^d)}$). 
We fix the integer $s\geq 200$ from now on.

\medskip
\subsection{Preparations for proving Lemma \ref{lemma T_{J,2}^s}: The partition of dyadic subcubes}
For $j\in\Z$, notice the definition
    \begin{align*}
    T_{J,2}^s b_{j-s}(x)
    = \int_{\R^d} \Omega_2(x-y) K_j^s(x,y) b_{j-s}(y) \chi_J(y) \dif y,\quad 
    J\in\mD_j. 
    \end{align*}
Hence, we only need to consider those cubes $J\in\mD_j$ for which $b_{j-s}\chi_J \neq 0$.
The following lemma estimates the total measure of all such cubes, which is essential for handling the sum $\sum\limits_{j\geq k} \sum\limits_{J\in\mD_j} T_{J,2}^s b_{j-s}$ in Lemma~\ref{lemma T_{J,2}^s}.

\begin{lemma} \label{sum|J|}
For any measurable set $A\subset \R^d$, we have
	\begin{align*}
	\sum_{j\in\Z} \sum_{{\begin{subarray}{c} J\in \mD_j, J\subseteq A \\ b_{j-s}\chi_J \neq 0 \end{subarray}}} |J| 
	\lesssim 2^{ds} |A|.
	\end{align*}
Moreover, 
	\begin{align*}
	\sum_{j\in\Z} \sum_{{\begin{subarray}{c} J\in \mD_j,  \\ b_{j-s}\chi_J \neq 0 \end{subarray}}} |J| 
	\lesssim 2^{ds} \lambda^{-1} C_{\Omega} \|f\|_{L^1(\R^d)}.
	\end{align*}
\end{lemma}

\begin{proof}
Fix a $j\in\Z$.
Since $b_{j-s} = \sum\limits_{Q\in\mQ_{j-s}} b_Q$ and $\operatorname{supp}(b_Q) \subseteq Q$ (see (cz-ii)), any cube $J\in\mD_j$ such that $b_{j-s}\chi_J \neq 0$ must intersect some $Q\in\mQ_{j-s}$.
Because the side length of $Q$ is much smaller than that of $J$, and both $Q$ and $J$ are dyadic cubes, we actually have $Q\subset J$.
Choose one such $Q$ and denote it by $Q_J$ (if more than one such $Q$ exists, we simply pick one).

We claim that all $Q_J$'s are disjoint.
Indeed, take two distinct dyadic cubes \(J\in\mathcal{D}_j\) and \(K\in\mathcal{D}_k\) with \(J\neq K\). 
If $k=j$, then $J \cap K=\varnothing$, hence $Q_J\cap Q_K=\varnothing$.
If $k\neq j$, then $\ell(Q_J) \neq \ell(Q_K) $, which means $Q_J\neq Q_K$. 
This gives $Q_J\cap Q_K=\varnothing$ by using (cz-iv).  

Now we turn to the proof of the lemma.
Observe that $|Q_J| = 2^{(j-s)d} = 2^{-sd}|J|$.
Using the disjointness of the cubes $Q_J$'s and applying (cz-v), we obtain
    \begin{align*}
    \sum_{j\in\Z} \sum_{{\begin{subarray}{c} J\in \mD_j, J\subseteq A \\ b_{j-s}\chi_J \neq 0 \end{subarray}}} |J| 
    &\leq 2^{-sd}  \sum_{j\in\Z} \sum_{{\begin{subarray}{c} J\in \mD_j, J\subseteq A \\ Q_J\in\mQ_{j-s} \\ Q_J\subset J \end{subarray}}} |Q_J| \\
    &\leq 2^{-sd}  \sum_{j\in\Z} \sum_{{\begin{subarray}{c} Q\in\mQ_{j-s} \\ Q\subseteq A \end{subarray}}} |Q| \\
    &\lesssim 2^{-sd} \min\big\{ |A|,\ \lambda^{-1} C_{\Omega} \|f\|_{L^1(\R^d)} \big\},
    \end{align*}
 which proves the desired estimates.
\end{proof}

Fix a cube $J\in\mD_j$. 
To partition $3J$ into smaller cubes, we divide each of its sides into three equal segments. 
This gives $3^d$ congruent subcubes, each of side length $\ell(J)/3$, with \(J\) itself as the central cube.
Label these subcubes as $J^1, J^2, \dots, J^{3^d}$.
Since $J\in\mD_j$, the cubes obtained by trisecting \(3J\) also lie in \(\mathcal{D}_j\). 
We call these \(J^m\) (\(1\le m\le 3^d\)) \emph{generated cubes} (as they are generated from the dyadic cube $J$), and we call $J$ the \emph{original cube} of $J^m$.

\begin{remark} \label{three possibilities}
Since all generated cubes lie in the dyadic grid \(\mathcal{D}\),  any two of them are either disjoint or one contains the other.  
In other word, for distinct \(J\in\mathcal{D}_j\) and \(K\in\mathcal{D}_k\), and any \(1\le m,n\le 3^d\), we have
\[
J^m \cap K^n \in \{J^m,\; K^n,\; \varnothing\}.
\]
\end{remark}

\textbf{The labeling is fixed} 
so that for any cube in $\mD = \bigcup\limits_{j\in\Z} \mD_j$, the generated cube occupying the same relative position always receives the same superscript (see Figure~\ref{fig:labeling}). 
As a result, once we know a generated cube \(J^m\), its superscript \(m\) tells us exactly where \(J\) is located relative to \(J^m\).
In particular, for a fixed \(m\in\{1,\dots,3^d\}\), the map \(J\mapsto J^m\) is one-to-one: different original cubes produce different generated cubes with the same superscript.
This property is essential for our subsequent proof, because it allows us to uniquely recover the original cube \(J\) from any generated cube \(J^m\) with a known superscript \(m\).

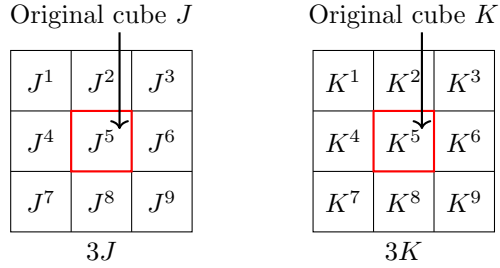
\begin{figure}[htbp]
\centering
\begin{tikzpicture}[scale=0.8, every node/.style={font=\small}]
\draw (0,0) rectangle (3,3);
\foreach \x in {1,2} \draw (\x,0) -- (\x,3);
\foreach \y in {1,2} \draw (0,\y) -- (3,\y);
\node at (0.5,2.5) {$J^1$};
\node at (1.5,2.5) {$J^2$};
\node at (2.5,2.5) {$J^3$};
\node at (0.5,1.5) {$J^4$};
\node at (1.5,1.5) {\textbf{$J^5$}};
\node at (2.5,1.5) {$J^6$};
\node at (0.5,0.5) {$J^7$};
\node at (1.5,0.5) {$J^8$};
\node at (2.5,0.5) {$J^9$};
\draw[thick,red] (1,1) rectangle (2,2);
\node[below] at (1.5,0) {$3J$};
\node[above] at (1.5,3.2) {Original cube $J$};

\begin{scope}[xshift=5cm]
\draw (0,0) rectangle (3,3);
\foreach \x in {1,2} \draw (\x,0) -- (\x,3);
\foreach \y in {1,2} \draw (0,\y) -- (3,\y);
\node at (0.5,2.5) {$K^1$};
\node at (1.5,2.5) {$K^2$};
\node at (2.5,2.5) {$K^3$};
\node at (0.5,1.5) {$K^4$};
\node at (1.5,1.5) {\textbf{$K^5$}};
\node at (2.5,1.5) {$K^6$};
\node at (0.5,0.5) {$K^7$};
\node at (1.5,0.5) {$K^8$};
\node at (2.5,0.5) {$K^9$};

\draw[thick,red] (1,1) rectangle (2,2);
\node[below] at (1.5,0) {$3K$};
\node[above] at (1.5,3.2) {Original cube $K$};
\end{scope}

\draw[->, thick] (1.8,3.3) -- (1.8,1.7) node[midway, above, sloped] {};
\draw[->, thick] (6.8,3.3) -- (6.8,1.7) node[midway, above, sloped] {};

\end{tikzpicture}
\caption{
Illustration of the fixed labeling: 
The generated cubes $J^m$ and $K^m$ occupy the same relative position inside $3J$ and $3K$, respectively. 
The central squares are the original cubes \(J\) and \(K\); they also serve as the generated cubes \(J^5\) and \(K^5\).
}\label{fig:labeling}
\end{figure}

\begin{definition}
Fix $1\leq m \leq 3^d$.
Let
	\begin{align*}
	H_{s,m}^0 = \{ J^m:\ J\in\mD_j,\ b_{j-s}\chi_J \neq 0,\ j\in\Z\}.	
	\end{align*}
For $i\geq 1$, define recursively
    \begin{align*}
	H^i_{s,m} = \{J^m\in H^{i-1}_{s,m}:\ \#(G^i_s(J^m)) \geq \lceil C_0 2^{ds} \rceil \},
    \end{align*}
where for $J^m\in H^{i-1}_{s,m}$,
	\begin{align*}
	G_s^i (J^m) = \{K^m\in H^{i-1}_{s,m}:\ J^m\subseteq K^m  \}.
    \end{align*}
Here, $C_0$ is a large constant to be chosen later.
\end{definition}

%
%
%
%
%

In plain words, \(H_{s,m}^0\) collects all generated cubes whose original cubes actually contribute to the bad function. 
For a cube \(J^m\in H_{s,m}^{i-1}\), the set \(G_s^i(J^m)\) consists of all the "ancestors" of \(J^m\) inside \(H_{s,m}^{i-1}\) (including itself). 
By Remark~\ref{three possibilities} with \(n=m\), these ancestor cubes are totally ordered by inclusion, forming a nested structure.
The condition \(\#G_s^i(J^m)\ge\lceil C_0 2^{ds}\rceil\) means that the containment chain below \(J^m\) contains at least that many distinct cubes. 
Thus \(H_{s,m}^i\) selects those cubes that lie sufficiently deep in such a chain.
Figure~\ref{fig:layered_structure} illustrates this nested structure.

\begin{figure}[htbp] 
\centering
\begin{tikzpicture}[
    x=1cm, y=1cm,
    line width = 1 pt,
    every node/.style={font=\small},
    >=stealth
]

\newcommand{\interval}[4]{
  \draw (#1,#3) -- (#2,#3);
  \draw (#1,#3-0.08) -- (#1,#3+0.08);
  \draw (#2,#3-0.08) -- (#2,#3+0.08);
  \ifx&#4& \else \node[anchor=west] at (#2+0.1,#3) {#4}; \fi
}


\node at (2.5,-0.25) {$\vdots$};

\interval{0.5}{4.5}{0}

\interval{2.5}{4.5}{0.5}{$K^m$}

\draw[line width=1.2pt, red] (2.5,1) -- (3.5,1);
\draw[red] (2.5,0.92) -- (2.5,1.08);
\draw[red] (3.5,0.92) -- (3.5,1.08);
\node[anchor=west, red] at (3.6, 1) {$J^m$ };

\interval{3}{3.5}{1.5}{}

\node at (3.25,1.95) {$\vdots$};

\interval{3.25}{3}{2.25}{}

\node at (3.125, 2.75) {$\vdots$};

\draw[decorate,decoration={brace, amplitude=3.5pt, mirror}, line width=1.0pt, red] 
    (5.4, -0.5) -- (5.4, 1) 
    node[midway, right, xshift = 2.5 pt] {$G_s^1(J^m)$};

\draw[decorate,decoration={brace, amplitude=3.5pt}, line width=1.0pt] 
    (0.25, -0.5) -- (0.25, 3) 
    node[midway, left, xshift = -2 pt] {$H^0_{s,m}$};


\node at (10,-0.25) {$\vdots$};    
    
\interval{9}{13}{0}    
    
\interval{9}{11}{0.5}
    
\interval{9}{10}{1}

\node at (9.75,1.55) {$\vdots$};

\interval{9.5}{10}{1.8}
    

\draw[line width=1.2pt, red] (9.75, 2.3) -- (10, 2.3);
\draw[red] (9.75, 2.22) -- (9.75, 2.38);
\draw[red] (10, 2.22) -- (10, 2.38);
\node[anchor=west, red] at (9.975, 2.3) {$L^m$};


\node at (9.86, 2.8) {$\vdots$};

\draw[decorate,decoration={brace, amplitude=3.5pt}, line width=1.0pt] 
    (7.75, -0.5) -- (7.75, 3) 
    node[midway, left, xshift = -1.8 pt] {$H^{i-2}_{s,m}$};
    
\draw[decorate,decoration={brace, amplitude=3.5pt}, line width=1.0pt] 
    (8.825, 0.5) -- (8.825, 3) 
    node[midway, left, xshift = -1.8 pt] {$H_{s,m}^{i-1}$};      

\draw[decorate,decoration={brace, amplitude=3.5pt, mirror}, line width=1.0pt,red] 
    (11.25, 0.5) -- (11.25, 2.3) 
    node[midway, right, xshift = 2.5 pt] {$G^{i}_{s}(L^m)$};    
    
\end{tikzpicture}

\caption{
Schematic of the layered structure in $H_{s,m}^i$, each horizontal segment represents a cube. 
For $J^m\in H^0_{s,m}$, by the inclusion relationship, these cubes $K^m$'s such that $K^m\in H^0_{s,m}$ and $K^m \supseteq J^m$ form a staircase, and $\#G_s^1(J^m)$ counts how many steps (floors) lie under $J^m$ in that staircase.  
If this number is at least $\lceil C_0 2^{ds}\rceil$ (i.e., the staircase is high enough), then $J^m$ belongs to $H_{s,m}^1$.
Similarly, for $i>1$ and $L^m \in H^{i-1}_{s,m}$, we select $L^m$ into $H^{i}_{s,m}$ if its staircase inside $H_{s,m}^{i-1}$ has at least $\lceil C_0 2^{ds}\rceil$ floors. 
} \label{fig:layered_structure}
\end{figure}
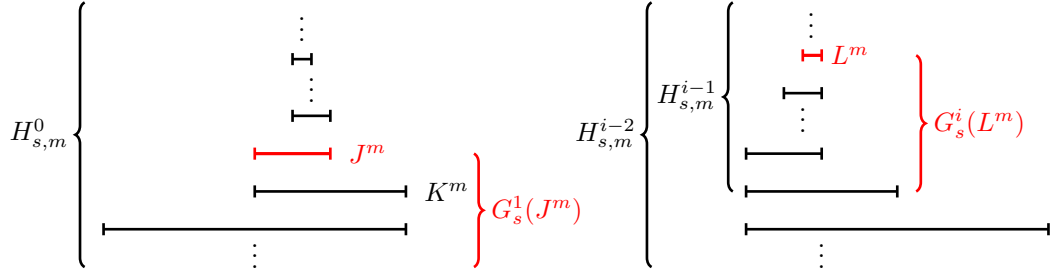

The following contents (especially Lemma \ref{|F^i_{s,m}|}) will be needed in the next section, where they are used to obtain an $L^1$ estimate for $b_Q$.
However, since the relevant framework has just been established, it is more natural to state them here rather than interrupt the later argument with a separate construction.

\begin{definition}\label{def of F^i_{s,m}}
Fix $1\leq m \leq 3^d$.
For $i\geq 0$, we set
	\begin{align*}
	F^i_{s,m} = \bigcup_{j\in\Z} \bigcup	_{{\begin{subarray}{c} J\in\mD_{j} \\ J^m\in H^i_{s,m} \end{subarray}}} J^m.
	\end{align*}
\end{definition}

\begin{remark} \label{inclusion relation of H^i_{s,m}, F^i_{s,m}}
From the definitions above, the families $\{H_{s,m}^i\}_{i\ge0}$ and $\{F_{s,m}^i\}_{i\ge0}$ satisfy the following nested inclusions:
    \begin{align*}
    & H^0_{s,m} \supseteq H^1_{s,m} \supseteq \cdots \supseteq H^i_{s,m} \cdots, \\
    & F^0_{s,m} \supseteq F^1_{s,m} \supseteq \cdots \supseteq F^i_{s,m} \cdots.
    \end{align*}
\end{remark}

Because the cubes in $H_{s,m}^i$ are either nested or disjoint (see Remark~\ref{three possibilities} with \(n=m\)), the union $F_{s,m}^i$ is exactly the union of the maximal cubes in $H_{s,m}^i$. 
Moreover, such maximal cubes exist: Lemma~\ref{sum|J|} together with the nested inclusion in Remark~\ref{inclusion relation of H^i_{s,m}, F^i_{s,m}} implies that the total measure of the cubes in $H_{s,m}^i$ is finite. 
In a collection of dyadic cubes with finite total measure, maximal cubes (with respect to inclusion) are always attained.  
Thus, we may write
    \begin{align*}
    F^i_{s,m} = \bigcup_{j\in\Z} \bigcup	_{{\begin{subarray}{c} J\in\mD_{j} \\ J^m\in \mH^i_{s,m} \end{subarray}}} J^m, 
    \end{align*}
where
    \begin{align*}
    \mH^i_{s,m} := \{\text{the maximal cubes in}\ H^i_{s,m}\}.
    \end{align*}
Obviously, these maximal cubes in $\mH^i_{s,m}$ are pairwise disjoint.

For a maximal cube \(J^m\in\mH_{s,m}^i\), we have the exact count
    \begin{align} \label{exactly C_0 2^{ds}}
    \#G_s^i(J^m)=\lceil C_0 2^{ds}\rceil.
    \end{align}
Indeed, by definition \(\#G_s^i(J^m)\ge\lceil C_0 2^{ds}\rceil\). 
If it were strictly larger, then there would exist a strictly larger cube \(K^m\in H_{s,m}^{i-1}\) also satisfying the condition, which would imply \(K^m\in H_{s,m}^i\) and contradict the maximality of \(J^m\).
We illustrate the construction of the family $\{ \mH^i_{s,m} \}_{i\geq 0}$ in Figure \ref{figure: H^i_{s,m}}.


\begin{figure}[htbp]
\centering
\begin{tikzpicture}[
    x=1cm, y=1cm,
    line width= 1 pt,
    every node/.style={font=\small},
    >=stealth
]

\newcommand{\interval}[4]{
  \draw (#1,#3) -- (#2,#3);
  \draw (#1,#3-0.08) -- (#1,#3+0.08);
  \draw (#2,#3-0.08) -- (#2,#3+0.08);
  \ifx&#4& \else \node[anchor=west] at (#2+0.1,#3) {#4}; \fi
}

\interval{1}{9}{0}{the maximal cube in $H^0_{s,m}$}
\draw[<-, line width=1.0pt] (0.2, 0) -- (0.2, 4.5) node[midway, left, align=center] {$H^0_{s,m}$ };
\draw[dashed] (0,0) -- (1,0);
\node at (3, 0.55) {$\vdots$};

\interval{1}{5}{1}{the maximal cube in $H^1_{s,m}$}
\draw[<-, line width=1.0pt] (0.8, 1) -- (0.8, 4.5) node[midway, right, align=center] {$H^1_{s,m}$ };
\draw[dashed] (0.5,1) -- (1,1);
\node at (4, 1.5) {$\vdots$};
\node at (4, 2) {$\vdots$};

\interval{3}{5}{2.25}{the maximal cube in $H^{i-1}_{s,m}$}
\draw[<-, line width=1.0pt] (2.25, 2.25) -- (2.25, 4.5) node[midway, right, align=center] {$H^{i-1}_{s,m}$ };
\draw[dashed] (1.95,2.25) -- (3,2.25);
\node at (4.5, 2.75) {$\vdots$};

\interval{4}{5}{3.25}{the maximal cube in $H^i_{s,m}$}
\draw[<-, line width=1.0pt] (3.5, 3.25) -- (3.5, 4.5) node[midway, right, align=center] {$H^i_{s,m}$ };
\draw[dashed] (3.25, 3.25) -- (4, 3.25);
\node at (4.5, 3.7) {$\vdots$};


\draw[decorate,decoration={brace, amplitude=3.5pt, mirror}, line width=1.0pt, red] 
    (9.3, 0.2) -- (9.3, 1.1) 
    node[midway, right, xshift=6pt] { \textbf{exactly} $\lceil C_0 2^{ds}\rceil$ \textbf{layers}};

\draw[decorate,decoration={brace, amplitude=3.5pt, mirror}, line width=1.0pt, red] 
    (9.36, 2.25) -- (9.36, 3.25) 
    node[midway, right, xshift=6pt] { \textbf{exactly} $\lceil C_0 2^{ds}\rceil$ \textbf{layers} };

\end{tikzpicture}
\caption
{
The bottom layer of each $H^i_{s,m}$ depends on the maximal cube.
We select a family of maximal cubes with a nested inclusion to illustrate this construction.
Then, the number of layers between two adjacent layers is exactly $\lceil C_0 2^{ds}\rceil$.
}\label{figure: H^i_{s,m}}
\end{figure}
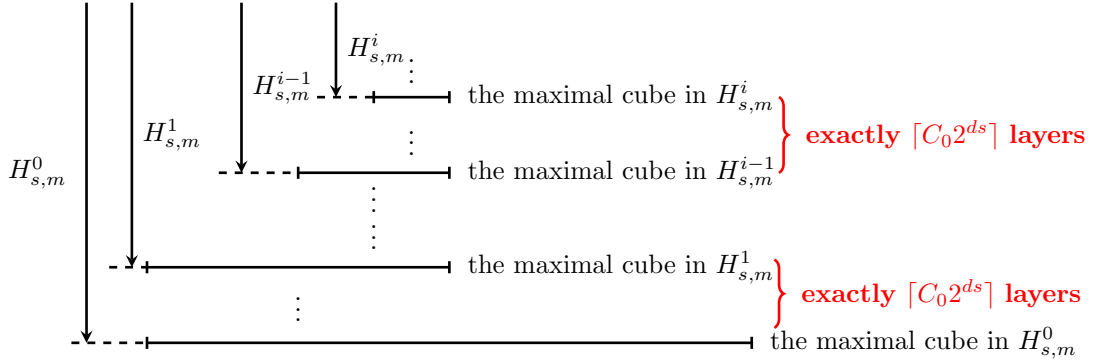

Having defined the nested families $\{F_{s,m}^i\}_{i\geq 0}$, we now prove an exponential decay estimate for their measures. This is a key geometric fact needed later.
\begin{lemma} \label{|F^i_{s,m}|}
Fix $1\leq m \leq 3^d$.
For $i\geq 1$,
    \begin{align*}
    |F^i_{s,m}| \lesssim 3^{-i} \lambda^{-1} C_{\Omega} \|f\|_{L^1(\R^d)}.
    \end{align*}
\end{lemma}

\begin{proof}
%
%
%
%

We first handle $i=1$.
For a cube $K \in H^0_{s,m}$, by $|K^m| = |K|$ and Lemma \ref{sum|J|}, we have
    \begin{align*}
    \sum_{K^m\in H^0_{s,m}} |K^m|
    = \sum_{k\in\Z} \sum_{{\begin{subarray}{c} K\in \mD_k,  \\ b_{k-s}\chi_K \neq 0 \end{subarray}}} |K| 
	\lesssim 2^{ds} \lambda^{-1} C_{\Omega} \|f\|_{L^1(\R^d)}.
	\end{align*}
This result, combined with (\ref{exactly C_0 2^{ds}}), gives
    \begin{align*}
    |F_{s,m}^1|
    = \sum_{j\in\Z} \sum_{{\begin{subarray}{c} J\in\mD_{j} \\ J^m\in \mH^1_{s,m} \end{subarray}}} |J^m| 
    \leq \lceil C_0 2^{ds} \rceil^{-1} \sum_{K^m\in H^0_{s,m} } |K^m|
    \lesssim 3^{-1} \lambda^{-1} C_{\Omega} \|f\|_{L^1(\R^d)},
    \end{align*}
provided $C_0$ is chosen large enough.

Now assume $i\geq 2$.
For any $K^m\in H^{i-1}_{s,m}$, we have $K^m \subseteq F^{i-1}_{s,m}$, so its original cube $K$ satisfies $K \subseteq 3K^m \subseteq 3F^{i-1}_{s,m}$.
Consequently, by $H^i_{s,m} \subseteq H^0_{s,m}$ (see Remark \ref{inclusion relation of H^i_{s,m}, F^i_{s,m}}), $|K^m| = |K|$, and Lemma \ref{sum|J|},   
    \begin{align*}
    \sum_{K^m\in H^{i-1}_{s,m} } |K^m|
    \leq \sum_{{\begin{subarray}{c} K^m \in H^0_{s,m} \\ K \subseteq 3F^{i-1}_{s,m}\end{subarray}}} |K^m|
    = \sum_{k\in\Z} \sum_{{\begin{subarray}{c} K\in \mD_k, K \subseteq 3F^{i-1}_{s,m} \\ b_{k-s}\chi_K \neq 0  \end{subarray}}} |K|  
    \lesssim 2^{ds} |3F^{i-1}_{s,m}|.
    \end{align*}
Therefore, 
    \begin{align*}
    |F_{s,m}^i|
    &= \sum_{j\in\Z} \sum_{J^m\in\mH^i_{s,m}} |J^m| 	
    \leq \lceil C_0 2^{ds} \rceil^{-1} \sum_{K^m\in H^{i-1}_{s,m} } |K^m|
    \lesssim C_0^{-1} 3^d |F^{i-1}_{s,m}| \\
    &\lesssim 3^{-i} \lambda^{-1} C_{\Omega} \|f\|_{L^1(\R^d)},
    \end{align*}
again with a sufficiently large $C_0$.
This completes the proof.
\end{proof}

\medskip
\subsection{Proof of Lemma \ref{lemma T_{J,2}^s}}
We now return to the proof of Lemma \ref{lemma T_{J,2}^s}.
From the nested inclusions of $\{H^i_{s,m}\}_{i\geq 0}$ in Remark~\ref{inclusion relation of H^i_{s,m}, F^i_{s,m}}, we have the disjoint partition
    \begin{align*}
    H^0_{s,m} = \bigcup_{i\geq 1} \big( H^{i-1}_{s,m} \backslash H^i_{s,m} \big).
    \end{align*}
Fix a $j\in\Z$. 
Using the support property \eqref{supp T_J^s h}, the partition of $3J$, and the above decomposition, we obtain
    \begin{align*}
    \sum_{J\in\mD_j} T_{J,2}^s b_{j-s}
    &= \sum_{1\leq m\leq 3^d} \sum_{i\geq 1} \sum_{{\begin{subarray}{c} J\in\mD_{j} \\ J^m\in H^{i-1}_{s,m}\backslash H^i_{s,m} \end{subarray}}} \big( T_{J,2}^s b_{j-s} \big)\chi_{J^m},
    \end{align*}
which implies
    \begin{align} \label{T_{J,2}^s b_{j-s}}
    \sup_{k\in\Z} \Big| \sum_{j\geq k} \sum_{J\in \mD_j} T_{J,2}^s b_{j-s} \Big|
    \leq \sum_{1\leq m\leq 3^d} \sum_{i\geq 1} \sup_{k\in\Z} \Big| \sum_{j\geq k} \sum_{{\begin{subarray}{c} J\in\mD_{j} \\ J^m\in H^{i-1}_{s,m}\backslash H^i_{s,m} \end{subarray}}} \big( T_{J,2}^s b_{j-s} \big)\chi_{J^m} \Big|.
    \end{align}
Thus, it remains to control the contribution from each layer \(H_{s,m}^{i-1}\setminus H_{s,m}^i\) separately. 
This is the content of the following lemma.

\begin{lemma} \label{lemma (T_{J_2}^s b_{j-s}) chi_{J^m}}
There exists $\delta>0$ and a positive integer $N_0$ such that for all $1\leq m \leq 3^d$ and $i\geq 1$, we have
    \begin{align*}
    \Big\| \sup_{k\in\Z} & \Big| \sum_{j\geq k} \sum_{{\begin{subarray}{c} J\in\mD_{j} \\ J^m\in H^{i-1}_{s,m}\backslash H^i_{s,m} \end{subarray}}} \big( T_{J,2}^s b_{j-s} \big)\chi_{J^m} \Big| \Big\|_{L^2(\R^d)} \\
    &\lesssim 3^{-\frac{i}{2}} 2^{-\delta s} s^{ \frac{N_0}{2} \gamma^{-1}+1}  \big( \lambda C_{\Omega} \|f\|_{L^1(\R^d)} \big)^{1/2}.
    \end{align*}	
\end{lemma}

We first assume Lemma \ref{lemma (T_{J_2}^s b_{j-s}) chi_{J^m}} holds, and use it to prove Lemma \ref{lemma T_{J,2}^s}.
Combining \eqref{T_{J,2}^s b_{j-s}} with Lemma~\ref{lemma (T_{J_2}^s b_{j-s}) chi_{J^m}}, we get
    \begin{align*}
    &\Big\| \sup_{k\in\Z} \Big| \sum_{j\geq k} \sum_{J\in\mD_j} T_{J,2}^s b_{j-s} \Big| \Big\|_{L^2(\R^d)} \\
	&\quad \leq \sum_{1\leq m\leq 3^d} \sum_{i\geq 1}  \Big\| \sup_{k\in\Z} \Big| \sum_{j\geq k} \sum_{{\begin{subarray}{c} J\in\mD_{j} \\ J^m\in H^{i-1}_{s,m}\backslash H^i_{s,m} \end{subarray}}} \big( T_{J,2}^s b_{j-s} \big)\chi_{J^m} \Big\|_{L^2(\R^d)}\\
	&\quad \lesssim \sum_{i\geq 1} 3^{-\frac{i}{2}} 2^{-\delta s} s^{ \frac{N_0}{2} \gamma^{-1}+1}  \big( \lambda C_{\Omega} \|f\|_{L^1(\R^d)} \big)^{1/2} \\ 
	&\quad \lesssim 2^{-\delta s} s^{ \frac{N_0}{2} \gamma^{-1}+1}  \big( \lambda C_{\Omega} \|f\|_{L^1(\R^d)} \big)^{1/2},
    \end{align*}
which gives the desired bound in Lemma \ref{lemma T_{J,2}^s}.

\medskip
\section{Proof of Lemma \ref{lemma (T_{J_2}^s b_{j-s}) chi_{J^m}}} \label{Section: Proof of Lemma {lemma (T_{J_2}^s b_{j-s}) chi_{J^m}}}
In this section, we fix $i\ge 1$ and decompose the set $H_{s,m}^{i-1}\setminus H_{s,m}^i$ into layers of maximal cubes. 
This yields a finite collection of layers, which allows us to decompose the operator 
in Lemma~\ref{lemma (T_{J_2}^s b_{j-s}) chi_{J^m}} according to these layers. 
Consequently, we can linearize the supremum in the maximal operator via the Rademacher-Menshov theorem (Lemma~\ref{Rademacher-Menshov}). 
The required estimate then reduces to controlling a linear combination of the operators $\beta_u$ (see Lemma~\ref{epsilon_u beta_u}). 

\medskip
\subsection{Linearization of the maximal operator}
Fix $i\ge 1$. 
We begin by decomposing the set $H_{s,m}^{i-1}\setminus H_{s,m}^i$ into layers of maximal cubes by the following constructions.
    \begin{definition}
	Set
	\begin{align*}
	&\mL_1 = \{\text{the maximal cubes in}\ H^{i-1}_{s,m} \backslash H^i_{s,m}\}.
	\end{align*}
For $u\geq 1$, define recursively
    \begin{align*}
	\mL_{u+1} = \{\text{the maximal cubes in}\ \big(H^{i-1}_{s,m} \backslash H^i_{s,m} \big) \backslash \bigcup_{v=1}^u \mL_v\}.
	\end{align*}
    \end{definition}

The following properties are immediate from the definition; we omit the straightforward verification.
    \begin{enumerate}[(i)]
    \item $\mL_u = \varnothing$ when $u> u_0 := \lceil C_0 2^{ds} \rceil$.
    \item $\bigcup\limits_{u=1}^{u_0} \mL_u = H^{i-1}_{s,m} \backslash H^i_{s,m} $.
    \item Cubes in $\mL_u$ ($1\le u\le u_0$) are disjoint. 
    \end{enumerate}
Using these properties, for any $x$ in the support of $\sum\limits_{j\ge k}\;\sum\limits_{\substack{J\in\mD_j \\ J^m\in H_{s,m}^{i-1}\setminus H_{s,m}^i}} 
    \bigl( T_{J,2}^s b_{j-s} \bigr)\chi_{J^m}$, there is at most one cube of $\mL_u$ containing $x$ for each $1\le u\le u_0$; if such a cube exists, we denote it by $J^m_{u,x}$. 
Thus, we obtain a nested sequence $\{J^m_{u,x}\}_{1\le u\le u_0}$ (where some indices may be missing) such that 
    \[
    x\in J^m_{u_0,x} \subseteq J^m_{u_0-1,x} \subseteq \cdots \subseteq J^m_{1,x},
    \quad \text{and} \quad
    \{J^m_{u,x}\}_{1\le u\le u_0}\subseteq H_{s,m}^{i-1}\setminus H_{s,m}^i.
    \]

\begin{figure}[htbp]
\centering
\begin{tikzpicture}[
    x=1cm, y=1cm,
    line width= 1 pt,
    every node/.style={font=\small},
    >=stealth
]

\newcommand{\interval}[3]{%
  \draw (#1,#3)--(#2,#3);
  \draw (#1,#3-0.08)--(#1,#3+0.08);
  \draw (#2,#3-0.08)--(#2,#3+0.08);
}

\node[left] at (0,2.5) {$\mathcal L_{u_0}$};
\node[left] at (0,1) {$\mathcal L_{2}$};
\node[left] at (0,0) {$\mathcal L_{1}$};

\interval{1}{1.4}{2.5}
\interval{2.2}{2.7}{2.5}

\draw[red] (3.2,2.5) -- (4.3,2.5);
\draw[red] (3.2,2.42) -- (3.2,2.58);
\draw[red] (4.3,2.42) -- (4.3,2.58);

\interval{6.0}{6.5}{2.5}
\interval{9.0}{9.3}{2.5}

\interval{1}{1.8}{1}

\draw[red] (2.3,1) -- (4.3,1);
\draw[red] (2.3,0.92) -- (2.3,1.08);
\draw[red] (4.3,0.92) -- (4.3,1.08);

\interval{6.0}{7.0}{1}
\interval{7.8}{8.3}{1}
\interval{9.0}{9.5}{1}

\draw[red] (1,0) -- (4.3,0);
\draw[red] (1,-0.08) -- (1,0.08);
\draw[red] (4.3,-0.08) -- (4.3,0.08);
\interval{6.0}{8.3}{0}
\interval{9.0}{10.0}{0}

\node[left] at (-0.2, 1.85) {$\vdots$};
\node at (2.65, 1.85) {$\vdots$};

\draw[dashed, red] (3.5, 2.5) -- (3.5, 0);
\node at (3.5, 0) {\textcolor{red}{$\bullet$}};
\node at (3.5, 1) {\textcolor{red}{$\bullet$}};
\node at (3.5, 2.5) {\textcolor{red}{$\bullet$}};
\node at (3.5, -0.25) {\textcolor{red}{$x$}};
\node at (4.85, 0) {\textcolor{red}{$J^m_{1,x}$}};
\node at (4.85, 1) {\textcolor{red}{$J^m_{2,x}$}};
\node at (4.85, 2.5) {\textcolor{red}{$J^m_{u_0,x}$}};

\end{tikzpicture}
\caption{
Illustration of the nested cubes \(J^m_{u,x}\). 
The layers \(\mL_1,\mL_2,\dots,\mL_{u_0}\) are shown as horizontal segments.
For a fixed point \(x\), the dots mark the positions of \(x\) and the corresponding cubes containing it.
}
\end{figure}
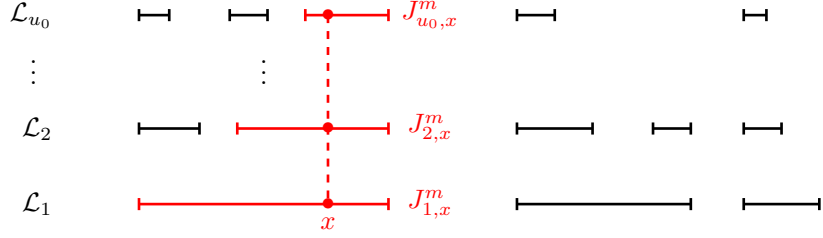

Hence,
    \begin{align*}
    \sup_{k\in\Z} \Big| \sum_{j\geq k} \sum_{{\begin{subarray}{c} J\in\mD_{j} \\ J^m\in H^{i-1}_{s,m}\backslash H^i_{s,m} \end{subarray}}} \big( T_{J,2}^s b_{j-s} \big)\chi_{J^m} (x) \Big| 
    = \sup_{1\leq v\leq u_0} \Big| \sum_{1\leq u\leq v} \big( T_{J_{u,x},2}^s b_{j-s} \big)\chi_{J_{u,x}^m} (x)  \Big|.
    \end{align*}
where $j_u$ is the integer such that the original cube $J$ of $J_{u,x}^m$ lies in $\mD_{j_u}$.

Now define for each $u=1,\dots,u_0$,     
    \begin{align*}
    \beta_u = \sum_{j\in\Z}  \sum_{{\begin{subarray}{c} J\in\mD_{j} \\ J^m\in \mL_u \end{subarray}}} \big( T_{J,2}^s b_{j-s} \big)\chi_{J^m}.
    \end{align*} 
Then the estimate required in Lemma~\ref{lemma (T_{J_2}^s b_{j-s}) chi_{J^m}} is equivalent to proving    
    \begin{align} \label{sum beta_u}
    \Big\| \sup_{1\leq v\leq u_0} \big| \sum_{1\leq u\leq v} \beta_u \big|  \Big\|_{L^2(\R^d)}
    \lesssim 3^{-\frac{i}{2}} 2^{-\delta s} s^{ \frac{N_0}{2} \gamma^{-1}+1}  \big( \lambda C_{\Omega} \|f\|_{L^1(\R^d)} \big)^{1/2},
    \end{align}     
which is a direct consequence of the following lemma together with the Rademacher-Menshov theorem (Lemma~\ref{Rademacher-Menshov}). 
\begin{lemma} \label{epsilon_u beta_u}
For any choice of $\epsilon_u\in\{-1, 1\}$ $(1\leq u\leq u_0)$, there exists a positive constant $\delta$ and a positive integer $N_0$ such that
    \begin{align*}
     \Big\| \sum_{1\leq u\leq u_0} \epsilon_u \beta_u  \Big\|_{L^2(\R^d)}
    \lesssim 3^{-\frac{i}{2}} 2^{-\delta s} s^{ \frac{N_0}{2} \gamma^{-1}}  \big( \lambda C_{\Omega} \|f\|_{L^1(\R^d)} \big)^{1/2}.
    \end{align*}
\end{lemma} 
To see this, we replace the constants $N$ and $B$ in Lemma~\ref{Rademacher-Menshov} by $u_0$ and $ 3^{-\frac{i}{2}} 2^{-\delta s} \cdot s^{ \frac{N_0}{2} \gamma^{-1}} \big(  \lambda C_{\Omega} \|f\|_{L^1(\R^d)} \big)^{1/2}$, respectively.
Since $\log(2+u_0)\lesssim_d s$, assuming Lemma \ref{epsilon_u beta_u} yields exactly (\ref{sum beta_u}).
Hence, we have reduced the proof of Lemma~\ref{lemma (T_{J_2}^s b_{j-s}) chi_{J^m}} to establishing Lemma~\ref{epsilon_u beta_u}, completing the linearization of the maximal operator.

\medskip
\subsection{Proof of Lemma \ref{epsilon_u beta_u}: Smoothing and microlocal decomposition}

By the definition of $\beta_u$ and the partition $\bigcup\limits_{u=1}^{u_0} \mL_u = H^{i-1}_{s,m} \backslash H^i_{s,m} $, we have
    \begin{align*}
    \sum_{1\leq u\leq u_0} \epsilon_u \beta_u 
    = \sum_{1\leq u\leq u_0}  \sum_{j\in\Z}  \sum_{{\begin{subarray}{c} J\in\mD_{j} \\ J^m\in \mL_u \end{subarray}}} \epsilon_u \big( T_{J,2}^s b_{j-s} \big)\chi_{J^m}
    = \sum_{j\in\Z}  \sum_{{\begin{subarray}{c} J\in\mD_{j} \\ J^m\in H^{i-1}_{s,m}\backslash H^i_{s,m} \end{subarray}}} \epsilon_u \big( T_{J,2}^s b_{j-s} \big)\chi_{J^m}.
    \end{align*} 
Notice that $b_{j-s} = \sum\limits_{Q\in\mQ_{j-s}} b_Q$.
Since $J^m\in H_{s,m}^{i-1}\setminus H_{s,m}^i$ implies $J\in\mD_j$ and $b_{j-s} \chi_J \neq 0$, only those \(Q\in\mQ_{j-s}\) with \(Q\subset J\) actually contribute to the sum.
We therefore rewrite the above equality as
    \begin{align*}
    \sum_{j\in\Z}  \sum_{{\begin{subarray}{c} J\in\mD_{j} \\ J^m\in H^{i-1}_{s,m}\backslash H^i_{s,m} \end{subarray}}}  \sum_{{\begin{subarray}{c} Q\in\mQ_{j-s} \\ Q\subseteq J \end{subarray}}}   \epsilon_u \big( T_{J,2}^s b_{Q} \big)\chi_{J^m} 
    &= \sum_{j\in\Z}   \sum_{Q\in\mQ_{j-s}}   \sum_{{\begin{subarray}{c} J\in\mD_{j} \\ J^m\in H^{i-1}_{s,m}\backslash H^i_{s,m} \\ Q\subseteq J \end{subarray}}}  \epsilon_u \big( T_{J,2}^s b_{Q} \big)\chi_{J^m} \\
    &= \sum_{j\in\Z}   \sum_{Q\in\mathfrak{Q}_{j-s}} \epsilon_{J_Q}  \big( T_{j} b_{Q} \big)\chi_{J_Q^m},
    \end{align*}
where the set $\mathfrak{Q}_{j-s}$ is defined as
    \begin{align*}
    &\qquad \{Q\in\mQ_{j-s}:  \text{there is exactly one } J\in\mD_j \ \text{such that}\ J^m\in H^{i-1}_{s,m} \backslash H^i_{s,m},\ Q\subset J\},
    \end{align*}
and we denote this unique $J$ associated with $Q$ by $J_Q$.

Therefore, it remains to prove the following lemma: 
\begin{lemma} \label{epsilon_{J_Q} (T_{j}^s b_{Q}) chi_{J^m}}
There exists a positive constant $\delta$ and a positive integer $N_0$ such that
	\begin{align*}
    \Big\| \sum_{j\in\Z} \sum_{Q\in\mathfrak{Q}_{j-s}} \epsilon_{J_Q}  \big( T_{j}^s b_{Q} \big)\chi_{J_Q^m}  \Big\|_{L^2(\R^d)}
    \lesssim 3^{-\frac{i}{2}} 2^{-\delta s} s^{ \frac{N_0}{2} \gamma^{-1}}  \big( \lambda C_{\Omega} \|f\|_{L^1(\R^d)} \big)^{1/2}.	
    \end{align*}
\end{lemma}

To prove this lemma, we first construct a smooth approximation of the characteristic function $\chi_{J^m}$ and then introduce a microlocal decomposition of the kernel.

Let $\omega\in C_c^{\infty}(\R^d)$ be a nonnegative, radial function satisfying $\supp\omega\subseteq\{x\in\R^d: |x|\leq 2^{-5}\}$ and $\int \omega =1$. 
Set $\omega_j(x)=2^{-jd}\omega(2^{-j}x)$.
Define the operator $P_j$ by $P_j h = \omega_j * h$.
Now, let 
    \begin{align*}
    \widetilde{\chi}_{J^m} = P_{j-s\kappa}(\chi_{J^m}),
    \end{align*}
where $\kappa\in (0,1)$ is a small constant to be chosen later.	
A straightforward calculation yields, for any multi-index \(\alpha\),
    \begin{align} \label{partial_x^alpha widetilde{chi}_{J^m} (x)}
    |\partial_x^{\alpha} \widetilde{\chi}_{J^m} (x) |
    \lesssim 2^{-(j-s\kappa)|\alpha|}.
    \end{align}

Next, we employ Seeger's modified microlocal decomposition method (see \cite{Seeger1996Singular}).
For a fixed $s\geq 200$, construct a collection $\Theta_s = \{e_v^s\}_v$ of unit vectors satisfying
    \begin{enumerate}[(i)]
    \item $|e_v^s - e_{v'}^s| \geq 2^{-s\mu-4}$ whenever $v\neq v'$,
    \item For every $\theta\in \sS^{d-1}$, there exists an unit vector $e^s_v\in\Theta_s$ such that $|e^s_v-\theta| \leq 2^{-s\mu-4}$,
    \end{enumerate}
where the constants $\mu$ and $\kappa$ satisfy $0<\mu<\kappa<1$.

Then, we can find a family of sets $\{E^s_v\}_v$ such that
    \begin{enumerate}[(i)]
    	\item $e^s_v \in E^s_v$.
    	\item $E^s_v \bigcap E^s_{v'} = \varnothing$ when $v\neq v'$.
    	\item $\bigcup\limits_v E^s_v = \sS^{d-1}$.
    	\item $2^{-s\mu-5} \leq \text{diam}(E^s_v) \leq 2^{-s\mu-2}$. 
    \end{enumerate}
Moreover, we have the following cardinality estimate:
    \begin{align} \label{Card Theta}
    \#\Theta_s \lesssim 2^{s\mu (d-1)}.
    \end{align}

Define 
    \begin{align*}
    K_j^{s,v} (x,y) = K_j^s(x,y) \chi_{E^s_v} \Big( \frac{x-y}{|x-y|} \Big),
    \end{align*} 
and the corresponding operator
    \begin{align*}
    T_j^{s,v} h(x) = \int_{\R^d} \Omega(x-y) K_j^{s,v}(x,y) h(y) \dif y.
    \end{align*}
Thus, the operator $T_j$ can be decomposed as $T_j = \sum\limits_v T_j^{s,v}$.

Let $\Phi$ be a smooth nonnegative radial function with $0\le\Phi\le1$, $\Phi(x)=1$ for $|x|\le2$ and $\Phi(x)=0$ for $|x|>4$.  
Define the Fourier multiplier operator $G_v^s$ by 
   \begin{align*}
    \mF(G^s_v h) (\xi) = \Phi ( 2^{s\mu} \langle e^s_v,\ \xi/|\xi| \rangle) \mF(h) (\xi).
    \end{align*}
One can verify that 
    \begin{align} \label{Phi(2^s delta <e^s_v, xi/|xi|>)}
    \sup_{\xi\neq 0} \sum_{v\in\Theta_s} \Phi^2 (2^{s\mu} \langle e^s_v,\ \xi/|\xi| \rangle ) 
    \lesssim 2^{s\mu (d-2)}.
    \end{align}

Using the smoothed characteristic function $\widetilde{\chi}_{J^m}$, the convolution operator $P_{j-s\kappa}$, the Fourier multiplier $G_v^s$, and the decomposition $T_j^s = \sum\limits_v T_j^{s,v}$, we split $\epsilon_{J_Q}(T_j^s b_Q)\chi_{J_Q^m}$ into four terms:
    \[
    \begin{aligned}
    I_1 &= \epsilon_{J_Q} (T_j^s b_Q)(\chi_{J_Q^m} - \widetilde{\chi}_{J_Q^m}),\\
    I_2 &= P_{j-s\kappa}\bigl( \epsilon_{J_Q} (T_j^s b_Q)\widetilde{\chi}_{J_Q^m} \bigr),\\
    I_3 &= \sum_v (I - P_{j-s\kappa}) G_v^s \bigl( \epsilon_{J_Q} (T_j^{s,v} b_Q)\widetilde{\chi}_{J_Q^m} \bigr),\\
    I_4 &= \sum_v (I - P_{j-s\kappa}) (I - G_v^s) \bigl( \epsilon_{J_Q} (T_j^{s,v} b_Q)\widetilde{\chi}_{J_Q^m} \bigr).
    \end{aligned}
    \]
Consequently, to prove Lemma~\ref{epsilon_{J_Q} (T_{j}^s b_{Q}) chi_{J^m}}, it suffices to establish the following four estimates for suitable positive constants $\delta_1$, $\delta_2$, $\delta_3$, $\delta_4$, and a positive integer $N_0$:
    \begin{align}
    \Bigl\| \sum_{j\in\Z}\sum_{Q\in\mathfrak{Q}_{j-s}} I_1 \Bigr\|_{L^2(\R^d)} &\lesssim 3^{-\frac{i}{2}} 2^{-\delta_1 s} \bigl( \lambda C_\Omega \|f\|_{L^1(\R^d)} \bigr)^{1/2}, \label{estimate for I_1}\\
    \Bigl\| \sum_{j\in\Z}\sum_{Q\in\mathfrak{Q}_{j-s}} I_2 \Bigr\|_{L^2(\R^d)} &\lesssim 3^{-\frac{i}{2}} 2^{-\delta_2 s} s^{\frac{1}{2} \gamma^{-1}} \bigl( \lambda C_\Omega \|f\|_{L^1(\R^d)} \bigr)^{1/2}, \label{estimate for I_2}\\
    \Bigl\| \sum_{j\in\Z}\sum_{Q\in\mathfrak{Q}_{j-s}} I_3 \Bigr\|_{L^2(\R^d)} &\lesssim 3^{-\frac{i}{2}} 2^{-\delta_3 s} \bigl( \lambda C_\Omega \|f\|_{L^1(\R^d)} \bigr)^{1/2}, \label{estimate for I_3}\\
    \Bigl\| \sum_{j\in\Z}\sum_{Q\in\mathfrak{Q}_{j-s}} I_4 \Bigr\|_{L^2(\R^d)} &\lesssim 3^{-\frac{i}{2}} 2^{-\delta_4 s} s^{\frac{N_0}{2}\gamma^{-1}} \bigl( \lambda C_\Omega \|f\|_{L^1(\R^d)} \bigr)^{1/2}. \label{estimate for I_4}
    \end{align}

The proofs of estimates (\ref{estimate for I_1})-(\ref{estimate for I_4}) are technically involved and rely on a careful analysis of oscillatory integrals. 
To keep the exposition focused, we postpone them to the next section.

At the end of this section, we provide an $L^1$ estimate for $b_Q$, which will be needed in the subsequent proofs.
From (cz-iii) and the definition of $\mathfrak{Q}_{j-s}$,
    \begin{align*}
    \sum_{j\in\Z} \sum_{Q\in\mathfrak{Q}_{j-s}} \| b_Q \|_{L^1(\R^d)}
    &\lesssim \lambda C_{\Omega}^{-1} \sum_{j\in\Z}  \sum_{Q\in\mathfrak{Q}_{j-s}} |Q| \\
    &\leq \lambda C_{\Omega}^{-1}\sum_{j\in\Z}  \sum_{Q\in\mathfrak{Q}_{j-s}} \min\Big\{ \Big| \bigcup_{j\in\Z} \bigcup_{Q\in\mQ_{j-s}} Q \Big|, \ \Big| \bigcup_{j\in\Z} \bigcup_{{\begin{subarray}{c} J\in\mD_{j} \\ J^m\in H^{i-1}_{s,m}\backslash H^i_{s,m} \end{subarray}}} J \Big| \Big\} \\
    &\leq \min\Big\{ \lambda^{-1}C_{\Omega} \|f\|_{L^1(\R^d)},\  \Big| \bigcup_{j\in\Z} \bigcup_{{\begin{subarray}{c} J\in\mD_{j} \\ J^m\in H^{i-1}_{s,m} \end{subarray}}}J \Big| \Big\}.
    \end{align*}    
Note that $J \subset 3J^m$. 
Using Lemma \ref{|F^i_{s,m}|},  
    \begin{align*}
    \Big| \bigcup_{j\in\Z} \bigcup_{{\begin{subarray}{c} J\in\mD_{j} \\ J^m\in H^{i-1}_{s,m} \end{subarray}}} J \Big|
    \leq 3^d \Big| \bigcup_{j\in\Z} \bigcup_{{\begin{subarray}{c} J\in\mD_{j} \\ J^m\in H^{i-1}_{s,m} \end{subarray}}} J^m \Big| 
    \lesssim |F^{i-1}_{s,m}|
    \lesssim 3^{-i+1} \lambda^{-1} C_{\Omega} \|f\|_{L^1(\R^d)}.
    \end{align*}
Finally, since $i\geq 1$, we obtain
    \begin{align} \label{sum|b_Q|_1}
    \sum_{j\in\Z} \sum_{Q\in\mathfrak{Q}_{j-s}} \| b_Q \|_{L^1(\R^d)} 
    \lesssim \min\big\{ \|f\|_{L^1(\R^d)},\ 3^{-i} \|f\|_{L^1(\R^d)} \min\big\}
    \leq 3^{-i } \|f\|_{L^1(\R^d)},
    \end{align}
which gives the bound for $b_Q$.

\medskip
\section{Proof of estimates (\ref{estimate for I_1})-(\ref{estimate for I_4})} \label{Section: Proof of estimates I_1-I_4}
In this section, we prove the four estimates \eqref{estimate for I_1}-\eqref{estimate for I_4}, which were used in the proof of Lemma~\ref{epsilon_{J_Q} (T_{j}^s b_{Q}) chi_{J^m}}. 
For \(I_1\) and \(I_2\), we establish \(L^1\) and \(L^3\) bounds and then interpolate between them. 
For \(I_3\), we use the Plancherel theorem, Cauchy-Schwarz inequality, and the orthogonality estimate \eqref{Phi(2^s delta <e^s_v, xi/|xi|>)} to reduce the proof to an \(L^2\) estimate, which is verified by a geometric argument involving the sets \(E_j^{s,v}\). 
For \(I_4\), the most complicated case, we need both an \(L^1\) estimate (derived via Littlewood-Paley decomposition and oscillatory integral estimates) and an \(L^3\) estimate; interpolation then yields the desired bound. 
The required technical lemmas are stated and proved at the end of this section.

\medskip
\subsection{Proof of (\ref{estimate for I_1})} 

We need the $L^1$ and $L^3$ estimates for $I_1$:
    \begin{align} \label{L_1 estimate for I_1}
    \Big\| \sum_{j\in\Z} \sum_{Q\in\mathfrak{Q}_{j-s}} I_1  \Big\|_{L^1(\R^d)}
    \lesssim 2^{-s\kappa} 3^{-i} \|\Omega\|_{\infty} \|f\|_{L^1(\R^d)}.
    \end{align}
    \begin{align} \label{L_3 estimate for I_1}
    \Big\| \sum_{j\in\Z} \sum_{Q\in\mathfrak{Q}_{j-s}} I_1  \Big\|_{L^3(\R^d)}
    \lesssim \|\Omega\|_{\infty} \big( \lambda^2 C_{\Omega}^{-2}  3^{-i} \|f\|_{L^1(\R^d)} \big)^{\frac{1}{3}}.
    \end{align}
Making an interpolation between (\ref{L_1 estimate for I_1}) and (\ref{L_3 estimate for I_1}), and using the fact that $\|\Omega\|_{\infty} \leq 2^{s\iota} \|\Omega\|_{1} \leq 2^{s\iota} C_{\Omega}$, we see
   \begin{align*}
   \text{LHS of}\ (\ref{estimate for I_1}) 
   &\lesssim \Big( 2^{-s\kappa} 3^{-i} \|\Omega\|_{\infty} \|f\|_{L^1(\R^d)} \Big)^{\frac{1}{4}}
   \cdot \Big( \|\Omega\|_{\infty} \big( \lambda^2 C_{\Omega}^{-2}  3^{-i} \|f\|_{L^1(\R^d)} \big)^{\frac{1}{3}} \Big)^{\frac{3}{4}}  \\
   &\lesssim 2^{-\frac{1}{4} s \kappa + s\iota} 3^{-\frac{i}{2}} \big( \lambda C_{\Omega} \|f\|_{L^1(\R^d)} \big)^{\frac{1}{2}}.
   \end{align*}
Choosing $\iota$ sufficiently small so that $\frac14\kappa > \iota$, and setting 
    \[
    \delta_1 := \frac14\kappa - \iota,
    \]
we obtain \eqref{estimate for I_1}.

\medskip
\noindent\textbf{Proof of the \(L^1\) estimate \eqref{L_1 estimate for I_1}.}
Fix a cube $Q\in\mathfrak{Q}_{j-s}$. 
For a $x\in Q$, we see
    \begin{align*}
    |(\chi_{J_Q^m} - \widetilde{\chi}_{J_Q^m}) (x)| 
    &\leq \int_{|u|\lesssim 2^{j-s\kappa}} |\chi_{J_Q^m}(x) - \chi_{J_Q^m}(x-u)|	 \omega_{j-s\kappa}(u) \dif u \\
    &\lesssim \chi_{ \{\dist(x, \partial J_Q^m)\lesssim 2^{j-s\kappa}\} }(x).
    \end{align*}
Combining this with \eqref{estimate for K_j^s(x,y)} gives
    \begin{align*}
     \big\| \big( T_{j}^s b_{Q} \big) (\chi_{J_Q^m} - \widetilde{\chi}_{J_Q^m}) \big\|_{L^1(\R^d)}
    &\lesssim 2^{-jd} \|\Omega\|_{\infty} \int_{\R^d} |b_Q(y)| \Big(\int_{\dist(x, \partial J_Q^m)\lesssim 2^{j-s\kappa}} \dif x \Big) \dif y\\
    &\lesssim 2^{-s\kappa} \|\Omega\|_{\infty} \|b_Q\|_{L^1(\R^d)}.
    \end{align*}
Summing over $j$ and $Q$ and using \eqref{sum|b_Q|_1} yields (\ref{L_1 estimate for I_1}).

\medskip
\noindent\textbf{Proof of the \(L^3\) estimate \eqref{L_3 estimate for I_1}.}
The $L^3$ estimate (\ref{L_3 estimate for I_1}) derives directly from the following lemma:
\begin{lemma} \label{sum sum |T^s_j b_Q|}
Using the notation established above, we have
    \begin{align*} 
    \Big\| \sum_{j\in\Z} \sum_{Q\in\mathfrak{Q}_{j-s}} |T_j^s b_Q|  \Big\|_{L^3(\R^d)}
    \lesssim \|\Omega\|_{\infty} \Big( \lambda^2 C_{\Omega}^{-2}  3^{-i} \|f\|_{L^1(\R^d)} \Big)^{1/3}.
    \end{align*}	
\end{lemma}

\begin{proof} 
For convenience, set $\mB_{m} = \sum\limits_{Q\in\mathfrak{Q}_{m}} b_Q$.
Then,
    \begin{align*}
    \Big\| \sum_{j\in\Z} \sum_{Q\in\mathfrak{Q}_{j-s}} |T_j^s b_Q|  \Big\|_{L^3(\R^d)}^3
    &\leq \int_{\R^d} \bigg( \sum_{j\in\Z} \int_{\R^d} |\Omega(x-y)| |K_j^s(x,y)| \mB_{j-s}(y) \dif y \bigg)^3 \dif x \\
    &\leq 3! \|\Omega\|_{\infty}^3 \sum_{j_1\geq j_2\geq j_3} \int_{\R^d} \prod_{t=1}^3 \bigg( \int_{\R^d} |K^s_{j_t}(x,y)| \mB_{j_t-s}(y) \dif y\bigg) \dif x,
    \end{align*}  
which is bounded by a constant multiple of
    \begin{align} \label{int int int int}
    \|\Omega\|_{\infty}^3 \sum_{j_1\geq j_2\geq j_3} \int_{\R^d} & \int_{\R^d} \int_{\R^d} \bigg( \int_{\R^d}  \prod_{t=1}^3 |K^s_{j_t}(x,y_t)| \dif x\bigg) \\
    & \cdot \mB_{j_1-s}(y_1) \mB_{j_2-s}(y_2) \mB_{j_3-s}(y_3) \dif y_3 \dif y_2 \dif y_1. \notag
    \end{align}
By (\ref{estimate for K_j^s(x,y)}) and the condition $j_1\geq j_2\geq j_3$, we get
    \begin{align*}
    |K^s_{j_1} (x, y_1)| \lesssim 2^{-j_1 d} \chi_{ \{|y_1-y_2|\leq 2^{j_1} \} },
    \quad |K^s_{j_2} (x, y_2)| \lesssim 2^{-j_2 d} \chi_{ \{|y_2-y_3|\leq 2^{j_2} \} },
    \end{align*}
    \begin{align*}
    |K^s_{j_3} (x, y_3)| \lesssim 2^{-j_3 d} \chi_{ \{|x-y_3|\leq 2^{j_3} \} },
    \end{align*}
which imply
    \begin{align*}
	\int_{\R^d}  \prod_{t=1}^3 |K^s_{j_t}(x,y_t)| \dif x
	\lesssim 2^{-(j_1 + j_2)d} \chi_{ \{|y_1-y_2|\leq 2^{j_1} \} } \chi_{ \{|y_2-y_3|\leq 2^{j_2} \} }.
    \end{align*}
Hence,
    \begin{align*}
    (\ref{int int int int})
    \lesssim \|\Omega\|_{\infty}^3 \sum_{j_1} 2^{-(j_1+j_2)d } & \int_{\R^d} \mB_{j_1-s}(y_1) \sum_{j_2\leq j_1} \int_{\R^d} \mB_{j_2-s}(y_2) \chi_{ \{|y_1-y_2|\leq 2^{j_1} \} } (y_2) \\
    &\sum_{j_3\leq j_2} \int_{\R^d} \mB_{j_3-s}(y_3) \chi_{ \{|y_2-y_3|\leq 2^{j_2} \} } (y_3) \dif y_3 \dif y_2 \dif y_1.
    \end{align*}

We next observe that
    \begin{align} \label{sum int B_{j_3-s} chi dy_3}
    \sum_{j_3\leq j_2} \int_{\R^d} \mB_{j_3-s}(y_3) \chi_{ \{|y_2-y_3|\leq 2^{j_2} \} } (y_3) \dif y_3
    \lesssim \lambda C_{\Omega}^{-1} 2^{j_2 d}.
    \end{align}
Indeed, using the definition of $\mB_m$, we estimate the left-hand side of (\ref{sum int B_{j_3-s} chi dy_3}) as
    \begin{align*}
    \sum_{j_3\leq j_2} \sum_{Q\in\mathfrak{Q}_{j-s}} \int_{|y_2-y_3|\lesssim 2^{j_2}} b_Q(y_3) \dif y_3 
    &\lesssim \lambda C_{\Omega}^{-1} \sum_{j_3\leq j_2}  \sum_{{\begin{subarray}{c} Q\in\mathfrak{Q}_{j-s} \\ Q\bigcap B_{2^{j-2}}(y_2)  \neq \varnothing \end{subarray}}} |Q| \\
    &\lesssim \lambda C_{\Omega}^{-1} |B_{2^{j-2}}(y_2)| \\
    &\lesssim \lambda C_{\Omega}^{-1} 2^{j_2 d},
    \end{align*}    
which gives (\ref{sum int B_{j_3-s} chi dy_3}).

Then, the same argument applies to the sum over \(j_2\le j_1\) and the integral over \(y_2\). 
Therefore,
    \begin{align*}
    (\ref{int int int int})
    &\lesssim (\lambda C_{\Omega})^2 \|\Omega\|_{\infty}^3 \sum_{j_1} \int_{\R^d} \mB_{j_1-s} (y_1) \dif y_1 
    \leq (\lambda C_{\Omega})^2 \|\Omega\|_{\infty}^3 \sum_{j} \sum_{Q\in\mathfrak{Q}_{j-s}} \|b_Q\|_{L^1(\R^d)} \\
    &\lesssim \|\Omega\|_{\infty}^3 \Big( \lambda^2 C_{\Omega}^2 3^{-i} \|f\|_{L^1(\R^d)} \Big),
    \end{align*}
where the last inequality follows from (\ref{sum|b_Q|_1}).	
\end{proof}

Now, since $|\chi_{J^m}| \leq 1$, $|\widetilde{\chi}_{J^m}| \leq 1$, and by Lemma \ref{sum sum |T^s_j b_Q|}, we have
    \begin{align*}
    \text{LHS of } (\ref{L_3 estimate for I_1})
    \lesssim \Big\| \sum_{j\in\Z} \sum_{Q\in\mathfrak{Q}_{j-s}} | T_{j}^s b_{Q} | \Big\|_{L^3(\R^d)}
    \lesssim \|\Omega\|_{\infty} \Big( \lambda^2 C_{\Omega}^{-2}  3^{-i} \|f\|_{L^1(\R^d)} \Big)^{1/3},
    \end{align*}
which proves (\ref{L_3 estimate for I_1}).

\medskip
\subsection{Proof of (\ref{estimate for I_2})}    
   
We need the following $L^1$ and $L^3$ estimates for $I_2$:
    \begin{align} \label{L_1 estimate for I_2}
    \Big\| \sum_{j\in\Z} \sum_{Q\in\mathfrak{Q}_{j-s}} I_2  \Big\|_{L^1(\R^d)}
    \lesssim (2^{-s(1-\kappa)} + 2^{-s}s^{2\gamma^{-1}}) 3^{-i} \|\Omega\|_{\infty} \|f\|_{L^1(\R^d)}.
    \end{align}
    \begin{align} \label{L_3 estimate for I_2}
    \Big\| \sum_{j\in\Z} \sum_{Q\in\mathfrak{Q}_{j-s}} I_2  \Big\|_{L^3(\R^d)}
    \lesssim \|\Omega\|_{\infty} \big( \lambda^2 C_{\Omega}^{-2}  3^{-i} \|f\|_{L^1(\R^d)} \big)^{1/3}.
    \end{align}       
Interpolating between (\ref{L_1 estimate for I_2}) and (\ref{L_3 estimate for I_2}), and using $\|\Omega\|_{\infty} \le 2^{s\iota}\|\Omega\|_{1} \le 2^{s\iota}\mathcal{C}_\Omega$, we have
    \begin{align*}
    &\text{LHS of } (\ref{estimate for I_2}) \\
    &\lesssim \Big( (2^{-s(1-\kappa)} + 2^{-s}s^{2\gamma^{-1}}) 3^{-i} \|\Omega\|_{\infty} \|f\|_{L^1(\R^d)} \Big)^{\frac{1}{4}} 
    \cdot \Big( \|\Omega\|_{\infty} \big( \lambda^2 C_{\Omega}^{-2}  3^{-i} \|f\|_{L^1(\R^d)} \big)^{\frac{1}{3}} \Big)^{\frac{3}{4}}  \\
    &\lesssim 3^{-\frac{i}{2}} \big( 2^{-s(\frac{1}{4}(1-\kappa)-\iota)} + 2^{-s(\frac{1}{4}-\iota)} s^{\frac{1}{2} \gamma^{-1}} \big) \big( \lambda C_{\Omega} \|f\|_{L^1(\R^d)} \big)^{\frac{1}{2}}.
    \end{align*}  
Choosing $\iota$ such that $\frac14(1-\kappa) > \iota$, and setting
    \begin{align*}
    \delta_2 
    :x= \min \big\{ \frac{1}{4}(1-\kappa)-\iota,\ \frac{1}{4}-\iota\big\},
    \end{align*}   
we obtain \eqref{estimate for I_2}.

\medskip
\noindent\textbf{Proof of the \(L^1\) estimate \eqref{L_1 estimate for I_2}.} 
From (\ref{sum|b_Q|_1}), it suffices to show
    \begin{align} \label{L_1 estimate for I_2-2}
    \| I_2 \|_{L^1(\R^d)}
    \lesssim (2^{-s(1-\kappa)} + 2^{-s}s^{2\gamma^{-1}}) \|\Omega\|_{\infty} \|b_Q\|_{L^1(\R^d)}.
    \end{align}
To prove this, for a fixed $x\in\R^d$, we observe that
    \begin{align*}
    I_2 (x) 
    &= \epsilon_{J_Q} \int_{\R^d} \omega_{j-s\kappa}(x-u) (T_j^sb_Q)(u) \widetilde{\chi}_{J_Q^m}(u) \dif u \\
    &= \epsilon_{J_Q} \int_Q b_Q(y) \Big( \int_{\R^d} \omega_{j-s\kappa}(x-u) \Omega(u-y) K_j^s(u,y) \widetilde{\chi}_{J_Q^m}(u) \dif u \Big) \dif y \\
    &= \epsilon_{J_Q} \int_Q b_Q(y) F_{j,s,\kappa}^x (y) \dif y,
    \end{align*}
where 
    \begin{align*}
    F_{j,s,\kappa}^x (y)
    := \int_{\R^d} \omega_{j-s\kappa}(x-y-z) \Omega(z) K_j^s(y+z,y) \widetilde{\chi}_{J_Q^m}(y+z) \dif z.
    \end{align*}
Using the cancellation property of $b_Q$ (see (cz-ii)) and the mean value theorem, we have
    \begin{align*}
    \int_Q b_Q(y) F_{j,s,\kappa}^x (y) \dif y
    &= \int_Q b_Q(y) \big(F_{j,s,\kappa}^x (y) - F_{j,s,\kappa}^x (c(Q)) \big) \dif y \\
    &= \int_Q b_Q(y) \Big( \int_0^1 \langle y-c(Q), \nabla (F_{j,s,\kappa}^x) (y_t) \rangle \dif t \Big) \dif y,
    \end{align*}
where $y_t := ty + (1-t)c(Q)$.
Then, a direct computation gives 
    \begin{align*}
    \int_{\R^d} | \nabla(F_{j,s,\kappa}^x) (y_t) | \dif x 
    \lesssim \|\Omega\|_{\infty} (2^{-j+s\kappa} + 2^{-j+s_{\gamma}}).
    \end{align*}
Combining this with \(|y-c(Q)|\lesssim 2^{j-s}\) and \eqref{2^{s_{gamma}}}, 
    \begin{align*}
    \| I_2 \|_{L^1(\R^d)}
    &\leq \int_{\R^d} \int_{Q} |b_Q(y)| \Big( \int_0^1 |y-c(Q)| |\nabla(F_{j,s,\kappa}^x) (y_t)| \dif t \Big) \dif y \dif x \\
    &\lesssim 2^{j-s} \int_Q |b_Q(y)| \Big(\int_{\R^d} \int_0^1  | \nabla(F_{j,s,\kappa}^x) (y_t) | \dif t \dif x \Big) \dif y \\
    &\lesssim (2^{-s(1-\kappa)} + 2^{-s}s^{2\gamma^{-1}}) \|\Omega\|_{\infty} \|b_Q\|_{L^1(\R^d)}.
    \end{align*}
This proves \eqref{L_1 estimate for I_2-2}, and hence \eqref{L_1 estimate for I_2}.

\medskip
\noindent\textbf{Proof of the \(L^3\) estimate \eqref{L_3 estimate for I_2}.}
The following lemma will be useful.
\begin{lemma} \label{sum sum P_{j-s kappa}(| T^s_j b_Q|)}
Using the notation established above, we have
    \begin{align*} 
    \Big\| \sum_{j\in\Z} \sum_{Q\in\mathfrak{Q}_{j-s}} P_{j-s\kappa} \big( | T_{j}^s b_{Q} | \big)  \Big\|_{L^3(\R^d)}
    \lesssim \|\Omega\|_{\infty} \Big( \lambda^2 C_{\Omega}^{-2}  3^{-i} \|f\|_{L^1(\R^d)} \Big)^{1/3}.
    \end{align*}	
\end{lemma}

\begin{proof}
Set $\mB_{m} = \sum\limits_{Q\in\mathfrak{Q}_{m}} b_Q$ as before.
Consequently,
    \begin{align*}
    &\Big\| \sum_{j\in\Z} \sum_{Q\in\mathfrak{Q}_{j-s}} P_{j-s\kappa} \big( | T_{j}^s b_{Q} | \big)  \Big\|_{L^3(\R^d)}^3 \\
    &\leq \int_{\R^d} \Big( \sum_{j\in\Z} \sum_{Q\in\mathfrak{Q}_{j-s}} \int_{\R^d} \Big( \int_{\R^d} \omega_{j-s\kappa}(x-y) |\Omega(y-z)| |K^s_j(y,z)| \dif y \Big) |b_Q(z)| \dif z\Big)^3 \dif x  \\
    &\lesssim \|\Omega\|_{\infty}^3 \sum_{j_1\geq j_2\geq j_3} \int_{\R^d} \prod_{t=1}^3 \bigg( \int_{\R^d} |W_{j_t}(x,y)| \mB_{j_t-s}(y) \dif y\bigg) \dif x,
    \end{align*}
where $W_j(x,z) := \int_{\R^d} |\omega_{j-s\kappa}(x-y)| |K_j^s(y,z)| \dif y $.	
The remainder of the argument is completely analogous to the proof of Lemma~\ref{sum sum |T^s_j b_Q|}; we omit the details to avoid repetition.

\end{proof}

Now, using $|\widetilde{\chi}_{J^m}|\le 1$ and Lemma \ref{sum sum P_{j-s kappa}(| T^s_j b_Q|)}, 
    \begin{align*}
    \text{LHS of } (\ref{L_3 estimate for I_2})
    \leq \Big\| \sum_{j\in\Z} \sum_{Q\in\mathfrak{Q}_{j-s}} P_{j-s\kappa} \big( | T_{j}^s b_{Q} | \big)  \Big\|_{L^3(\R^d)}
    \lesssim \|\Omega\|_{\infty} \Big( \lambda^2 C_{\Omega}^{-2}  3^{-i} \|f\|_{L^1(\R^d)} \Big)^{1/3},
    \end{align*}
completing the proof of (\ref{L_3 estimate for I_2}).

\medskip
\subsection{Proof of (\ref{estimate for I_3})}

Applying the Plancherel theorem, the Cauchy-Schwarz inequality, \eqref{Phi(2^s delta <e^s_v, xi/|xi|>)}, and then Plancherel again, we see
    \begin{align*}
    &\Big\| \sum_{j\in\Z} \sum_{Q\in\mathfrak{Q}_{j-s}} I_3  \Big\|_{L^2(\R^d)}^2 \\
    &\quad = \int_{\R^d} \Big|  \sum_v \Phi ( 2^{s\mu} \langle e^s_v,\ \xi/|\xi| \rangle ) \mF\Big( \sum_{j\in\Z} \sum_{Q\in\mathfrak{Q}_{j-s}} \big(I - P_{j-s\kappa} \big) \big( \epsilon_{J_Q}  \big( T_{j}^{s,v} b_{Q} \big)\widetilde{\chi}_{J_Q^m} \big) \Big) (\xi)  \Big|^2 \dif\xi	\\
    &\quad \lesssim 2^{s\mu(d-2)} \Big\| \sum_v \Big| \mF\Big( \sum_{j\in\Z} \sum_{Q\in\mathfrak{Q}_{j-s}} \big(I - P_{j-s\kappa} \big) \big( \epsilon_{J_Q}  \big( T_{j}^{s,v} b_{Q} \big)\widetilde{\chi}_{J_Q^m} \big) \Big) \Big|^2 \Big\|_{L^1(\R^d)} \\
    &\quad \lesssim 2^{s\mu(d-2)} \sum_v  \Big\| \sum_{j\in\Z} \sum_{Q\in\mathfrak{Q}_{j-s}} \big(I - P_{j-s\kappa} \big) \big( \epsilon_{J_Q}  \big( T_{j}^{s,v} b_{Q} \big)\widetilde{\chi}_{J_Q^m} \big) \Big\|_{L^2(\R^d)}.
    \end{align*}
We now state the following estimate, which will be proved below; it is the key ingredient for establishing \eqref{estimate for I_3}.   
    \begin{equation} \label{(I-P)(epsilon (T b_Q)chi)}
    \begin{aligned}
    \Big\| \sum_{j\in\Z} \sum_{Q\in\mathfrak{Q}_{j-s}} \big(I -  P_{j-s\kappa} \big) & \big( \epsilon_{J_Q}  \big( T_{j}^{s,v} b_{Q} \big)\widetilde{\chi}_{J_Q^m} \big) \Big\|_{L^2(\R^d)}^2 \\
    &\lesssim \lambda C_{\Omega} 2^{2s\iota-2s\mu (d-1)} 3^{-i} \|f\|_{L^1(\R^d)}.
    \end{aligned}
    \end{equation}   
Inserting (\ref{(I-P)(epsilon (T b_Q)chi)}) into the previous inequality and using \eqref{Card Theta}, we get   
    \begin{align*}
    \Big\| \sum_{j\in\Z} \sum_{Q\in\mathfrak{Q}_{j-s}} I_3  \Big\|_{L^2(\R^d)}^2
    \lesssim 2^{-s\mu + 2s\iota} 3^{-i} \|f\|_{L^1(\R^d)},
    \end{align*}    
Then, by choosing \(\mu > 2\iota > 0\) and setting 
    \[
    \delta_3 := \frac{1}{2} (\mu - 2\iota),
    \]
we obtain \eqref{estimate for I_3}.

It remains to prove (\ref{(I-P)(epsilon (T b_Q)chi)}).
First, it follows from (\ref{estimate for K_j^s(x,y)}) that
    \begin{align*}
    |\big( T_{j}^{s,v} b_{Q} \big)(x)|
    &\leq \int_{\R^d} \big| \Omega(x-y) K_j^s(x,y) \chi_{E^s_v} \Big( \frac{x-y}{|x-y|} \Big) b_Q(y) \big| \dif y \\
    &\leq \|\Omega\|_{\infty} 2^{-jd} \int_{\R^d} \chi_{\{2^{j-5}\leq |x-y| \leq 2^{j-1}\}} \chi_{E^s_v} \Big( \frac{x-y}{|x-y|} \Big) |b_Q(y)| \dif y \\
    &\leq \|\Omega\|_{\infty} 2^{-jd} \chi_{E^{s,v}_j} * |b_Q|(x),
    \end{align*}
where
    \begin{align*}
    E^{s,v}_j := \{x\in\R^d: | \langle x,e^s_v \rangle |\leq 2^{j-1},\ |x-\langle x,e^s_v \rangle e^s_v|\leq 2^{j-s\mu-1}\},
    \end{align*}
and the volume of this set satisfies
    \begin{align} \label{|E^{s,v}_j|}
    |E^{s,v}_j| \lesssim 2^{jd-s\mu(d-1)}.
    \end{align}
Consequently, by the estimate for $ |\big( T_{j}^{s,v} b_{Q} \big)(x)|$ and $\int \omega_{j-s\kappa} =1$,
    \begin{equation} \label{(I-P)(epsilon(T b_Q)chi)}
    \begin{split}
    \big| \big(I - P_{j-s\kappa} \big) \big( \epsilon_{J_Q}  \big( T_{j}^{s,v} b_{Q} \big)\widetilde{\chi}_{J_Q^m} \big)(x) \big| 
    &\lesssim |\big( T_{j}^{s,v} b_{Q} \big)(x)| + |\big( \omega_{j-s\kappa} *\big( T_{j}^{s,v} b_{Q} \big) \big)(x)| \\
    & \lesssim \|\Omega\|_{\infty} 2^{-jd} \chi_{E^{s,v}_j} * |b_Q|(x) 
    \end{split}
    \end{equation}  
Hence, for a fixed unit vector $e^s_v\in\Theta_s$,
    \begin{align*} 
    \text{LHS of } (\ref{(I-P)(epsilon (T b_Q)chi)})
    &\lesssim \|\Omega\|_{\infty}^2 \int_{\R^d} \Big( \sum_{j\in\Z} 2^{-jd} \chi_{E^{s,v}_j}* \mB_{j-s}(x) \Big)^2 \dif x \\
    &= \|\Omega\|_{\infty}^2 \sum_{j\in\Z} \Big( 2^{-jd} \int_{\R^d} \big( \chi_{E^{s,v}_j} * \chi_{E^{s,v}_j} * \mB_{j-s} \big) \mB_{j-s} (x) \dif x \\
    &\quad \quad + 2 \sum_{i<j} 2^{-(i+j)d} \int_{\R^d} \big( \chi_{E^{s,v}_j} * \chi_{E^{s,v}_i} * \mB_{i-s} \big) \mB_{j-s} (x) \dif x \Big),
    \end{align*}
where \(\mathcal{B}_m = \sum\limits_{Q\in\mathfrak{Q}_m} |b_Q|\). 
Now, fix a $x\in\R^d$.
For $i\leq j$, a direct computation gives
    \begin{align*}
    \chi_{E^{s,v}_j} * \chi_{E^{s,v}_i} (x)
    &=  \chi_{E^{s,v}_j + E^{s,v}_i} (x) \int_{\R^d} \chi_{E^{s,v}_j}(y) \chi_{E^{s,v}_i}(x-y) \dif y \\
    &\lesssim 2^{id-s\mu(d-1)} \chi_{\widetilde{E}^{s,v}_j}(x),
    \end{align*}
where 
    \begin{align*}
    \widetilde{E}^{s,v}_j
    := E^{s,v}_j + E^{s,v}_j
    = \{u+v: u, v\in E^{s,v}_j\}.
    \end{align*}
Therefore,
    \begin{align*}
    \text{LHS of } (\ref{(I-P)(epsilon (T b_Q)chi)}) 
    \lesssim \|\Omega\|_{\infty}^2 \sum_{j\in\Z} 2^{-jd-s\mu (d-1)} \int_{\R^d} \sum_{i\leq j} \big( \chi_{\widetilde{E}^{s,v}_j} * \mB_{i-s} \big)(x) \mB_{j-s}(x) \dif x.
    \end{align*}
Using (cz-ii) and the volume estimate (\ref{|E^{s,v}_j|}), for any fixed \(x \in \R^d \),
    \begin{align*}
    \sum_{i\leq j} \big( \chi_{\widetilde{E}^{s,v}_j} * \mB_{i-s} \big)(x)
    &= \sum_{i\leq j} \sum_{{\begin{subarray}{c} Q\in\mathfrak{Q}_{j-s} \\ x-\widetilde{E}^{s,v}_j \bigcap Q \neq \varnothing \end{subarray}}} \int_Q |b_Q(y)| \dif y 
    \lesssim \lambda C_{\Omega}^{-1} \sum_{i\leq j} \sum_{{\begin{subarray}{c} Q\in\mathfrak{Q}_{j-s} \\ x-\widetilde{E}^{s,v}_j \bigcap Q \neq \varnothing \end{subarray}}} |Q| \\
    &\lesssim \lambda C_{\Omega}^{-1} |3\widetilde{E}^{s,v}_j| \\
    &\lesssim \lambda C_{\Omega}^{-1} 2^{jd-s\mu(d-1)}.
    \end{align*}
Finally, combining this with \(\|\Omega\|_{\infty} \le 2^{s\iota}\|\Omega\|_{1} \le 2^{s\iota}\mathcal{C}_\Omega\) and \eqref{sum|b_Q|_1}, we see
    \begin{align*}
    \text{LHS of } (\ref{(I-P)(epsilon (T b_Q)chi)})
    \lesssim \lambda C_{\Omega} 2^{2s\iota-2s\mu (d-1)} \sum_{j\in\Z} \int_{\R^d} \mB_{j-s} (x) \dif x
    \lesssim \lambda C_{\Omega} 2^{2s\iota-2s\mu (d-1)} 3^{-i} \|f\|_{L^1(\R^d)}.
    \end{align*}

\medskip
\subsection{Proof of (\ref{estimate for I_4})}

We assume the following $L^1$ estimate holds:
    \begin{equation}
    \begin{aligned} \label{L_1 estimate for I_4}
    \Big\| \sum_{j\in\Z} \sum_{Q\in\mathfrak{Q}_{j-s}} I_4  \Big\|_{L^1(\R^d)} 
    &\lesssim (C_{N_0} s^{2\gamma^{-1}N_0} 2^{-\beta_1 s} + 2^{-\beta_2 s}  \\
    &\quad \quad + s^{2\gamma^{-1}} 2^{-\beta_3 s} + 2^{-\beta_4 s}) 3^{-i}C_{\Omega} \|f\|_{L^1(\R^d)},	
    \end{aligned}
    \end{equation}
where \(N_0>0\) is a sufficiently large integer, and the exponents $\beta_1$, $\beta_2$, $\beta_3$, $\beta_4$ are defined by
    \[
    \beta_1 := \sigma N_0 - \big(\kappa N_0 + \mu N_0 + 2\mu ( \lfloor d/2 \rfloor + 1) + \iota \big), \qquad
    \beta_2 := 1-(2\sigma + \iota - \kappa),
    \]
    \[
    \beta_3 := 1 - (\sigma + \iota - \kappa),\qquad
    \beta_4 := 1 - (\sigma + \iota).
    \]
Here \(\sigma\in(0,1)\) is a parameter to be chosen later;
the constants \(\kappa\), \(\mu\), \(\iota\), and \(\gamma\) have already been introduced in previous sections.
They satisfy
    \begin{align} \label{0 < iota << mu << kappa << sigma < 1}
    0 < \iota \ll \mu \ll \kappa \ll \sigma \ll 1.
    \end{align}

We also need the following $L^3$ estimate:
    \begin{align} \label{L_3 estimate for I_4}
    \Big\| \sum_{j\in\Z} \sum_{Q\in\mathfrak{Q}_{j-s}} I_4  \Big\|_{L^3(\R^d)}
    \lesssim 2^{s\mu ( \lfloor d/2 \rfloor +1 ) +\frac{2}{3}s\mu (d-1)} \|\Omega\|_{\infty} \big( \lambda^2 C_{\Omega}^{-2}  3^{-i} \|f\|_{L^1(\R^d)} \big)^{1/3}.	
    \end{align}

Interpolating between (\ref{L_1 estimate for I_4}) and (\ref{L_3 estimate for I_4}), and using $\|\Omega\|_{\infty} \le 2^{s\iota}\|\Omega\|_{L^1}\le 2^{s\iota}\mathcal{C}_\Omega$, we have 
    \begin{align*}
    \text{LHS of } (\ref{estimate for I_4})
    &\lesssim \Big( (C_{N_0} s^{2\gamma^{-1}N_0} 2^{-\beta_1 s} + 2^{-\beta_2 s} + s^{2\gamma^{-1}} 2^{-\beta_3 s} + 2^{-\beta_4 s}) 3^{-i} C_{\Omega} \|f\|_{L^1(\R^d)}\Big)^{\frac{1}{4}} \\
    &\quad \quad \cdot \Big( 2^{s\mu(\lfloor d/2 \rfloor+1) +\frac{2}{3}s\mu (d-1) + s\iota}   C_{\Omega} \big( \lambda^2 C_{\Omega}^{-2}  3^{-i} \|f\|_{L^1(\R^d)} \big)^{\frac{1}{3}}\Big)^{\frac{3}{4}} \\
    &\lesssim \big( C_{N_0}s^{\frac{1}{2}N_0\gamma^{-1}} 2^{-\rho_1 s} + 2^{-\rho_2 s} + s^{\frac{1}{2} \gamma^{-1} } 2^{-\rho_3 s}  + 2^{-\rho_4 s}\big) \big( \lambda C_{\Omega} 3^{-i} \|f\|_{L^1(\R^d)} \big)^{\frac{1}{2}},
    \end{align*}
where
    \begin{align*} 
    \rho_1 := \frac{1}{4} \Big( \sigma N_0 - \big( \kappa N_0 + \mu N_0 + 2\mu ([d/2] + 1) + \iota \big) \Big) - \frac{3}{4} \big( \mu( \lfloor d/2 \rfloor ) +\frac{2}{3}\mu (d-1) + \iota \big), 
    \end{align*}
    \begin{align*} 
    \rho_2 := \frac{1}{4} \big( 1-(2 \sigma + \iota - \kappa) \big) - \frac{3}{4} \Big( \mu ( \lfloor d/2 \rfloor + 1) +\frac{2}{3}\mu (d-1) + \iota \Big), 
    \end{align*}
    \begin{align*} 
    \rho_3 := \frac{1}{4} \big( 1 - ( \sigma + \iota - \kappa) \big) - \frac{3}{4} \Big(s\mu( \lfloor d/2 \rfloor + 1) + \frac{2}{3} \mu (d-1) + \iota \Big), 
    \end{align*}
    \begin{align*} 
    \rho_4 := \frac{1}{4} \big( 1 - ( \sigma + \iota) \big) - \frac{3}{4} \Big( \mu( \lfloor d/2 \rfloor + 1 ) +\frac{2}{3} \mu (d-1) + \iota \Big).
    \end{align*}
The parameter condition \eqref{0 < iota << mu << kappa << sigma < 1}, together with \(N_0\) chosen sufficiently large, ensures that \(\rho_1,\rho_2,\rho_3,\rho_4 > 0\).

Finally, by setting
    \[
    \delta_4 := \min\{\rho_1,\rho_2,\rho_3,\rho_4\},
    \]
we obtain the desired estimate \eqref{estimate for I_4}.

We now turn to the proofs of the estimates used above.
We first prove the \(L^3\) estimate \eqref{L_3 estimate for I_4}, since the \(L^1\) estimate \eqref{L_1 estimate for I_4} is more involved.

\medskip
\noindent\textbf{Proof of the \(L^3\) estimate \eqref{L_3 estimate for I_4}.}
A direct calculation gives, for any multi-index \(\alpha\),
    \[
    |\partial_\xi^\alpha (1-\Phi(2^{s\mu}\langle e_v^s,\xi/|\xi|\rangle))|
    \lesssim 2^{s\mu|\alpha|}|\xi|^{-|\alpha|}.
    \]
By taking \(|\alpha|\le \lfloor d/2\rfloor+1\), the right-hand side is bounded by \( 2^{s\mu(\lfloor d/2\rfloor+1)}|\xi|^{-|\alpha|} \).
Since \(I-G_v^s\) is a Fourier multiplier with symbol \(1-\Phi(2^{s\mu}\langle e_v^s,\xi/|\xi|\rangle)\), the multiplier satisfies the hypotheses of Lemma~\ref{bounded T_m} with constant \(M = C 2^{s\mu(\lfloor d/2\rfloor+1)}\). 
Consequently, \(I-G_v^s\) is bounded on \(L^p\) for all \(1<p<\infty\). 
In particular, for $p=3$,
    \begin{align*} 
    \|I-G^s_v \|_{L^3 \rightarrow L^3}
    \lesssim 2^{s\mu ( \lfloor d/2 \rfloor +1)}.
    \end{align*}
By the triangle inequality, the above $L^3$ bound for \(I-G_v^s\), and the pointwise estimate \eqref{(I-P)(epsilon(T b_Q)chi)}, we obtain
    \begin{align*}
    \text{LHS of } (\ref{L_3 estimate for I_4})
    &\lesssim 2^{s\mu ( \lfloor d/2 \rfloor +1)} \sum_v \Big\|  \sum_{j\in\Z} \sum_{Q\in\mathfrak{Q}_{j-s}} \big(I - P_{j-s\kappa} \big) \big( \epsilon_{J_Q}  \big( T_{j}^{s,v} b_{Q} \big)\widetilde{\chi}_{J_Q^m} \big)  \Big\|_{L^3(\R^d)} \\
    &\lesssim 2^{s\mu ( \lfloor d/2 \rfloor +1)} \|\Omega\|_{\infty} \sum_v \Big\|  \sum_{j\in\Z} 2^{-jd} \chi_{E^{s,v}_j}* \mB_{j-s}  \Big\|_{L^3(\R^d)},
    \end{align*}
where $\mathcal{B}_m = \sum\limits_{Q\in\mathfrak{Q}_m} |b_Q|$.

We now estimate the \(L^3\) norm in the sum above.
Using the definition of convolution,
    \begin{align*}
    \Big\|  \sum_{j\in\Z} 2^{-jd} \chi_{E^{s,v}_j}* \mB_{j-s}  \Big\|_{L^3(\R^d)} 
    &\lesssim \sum_{j_1\geq j_2\geq j_3} \int_{\R^d} \prod_{t=1}^3 2^{-j_t d} (\chi_{E^{s,v}_{j_t} * \mB_{j_t-s}} ) (x) \dif x \\
    &= \sum_{j_1\geq j_2\geq j_3} \int_{\R^d} \int_{\R^d} \int_{\R^d} \Big(\int_{\R^d} \prod_{t=1}^3 2^{-j_t d} \chi_{E^{s,v}_{j_t}}(x-y_t) \dif x \Big) \\
    &\quad \quad \quad \quad \quad \cdot \mB_{j_1-s}(y_1) \mB_{j_2-s}(y_2) \mB_{j_3-s}(y_3) \dif y_1 \dif y_2 \dif y_3.
    \end{align*}
Since $E^{s,v}_j \subseteq \{x\in\R^d: |x|\lesssim 2^j \}$, the support conditions imply $|x-y_t|\lesssim 2^j$, $j=1,2,3$. 
Using the volume estimate \eqref{|E^{s,v}_j|}, we have
    \begin{align*}
    &\int_{\R^d} \prod_{t=1}^3 2^{-j_t d} \chi_{E^{s,v}_{j_t}}(x-y_t) \dif x \\
    &\quad \lesssim 2^{-(j_1+j_2+j_3)d} \chi_{\{|y_1-y_2|\lesssim 2^{j_1}\}}(y_1) \chi_{\{|y_2-y_3|\lesssim 2^{j_2}\}}(y_2) \int_{\R^d} \chi_{E^{s,v}_{j_3}(x-y_3)} \dif x \\
    &\quad \lesssim 2^{-(j_1+j_2)d} 2^{-s\mu(d-1)} \chi_{\{|y_1-y_2|\lesssim 2^{j_1}\}}(y_1) \chi_{\{|y_2-y_3|\lesssim 2^{j_2}\}}(y_2).
    \end{align*}
Proceeding as in the proof of Lemma~\ref{sum sum |T^s_j b_Q|} (see \eqref{sum int B_{j_3-s} chi dy_3} and the subsequent steps), we obtain    
    \begin{align*}
    &\Big\| \sum_{j\in\Z} 2^{-jd} \chi_{E^{s,v}_j}* \mB_{j-s}  \Big\|_{L^3(\R^d)}^3 \\
    &\quad \lesssim 2^{-s\mu(d-1)} \sum_{j_1} 2^{-(j_1+j_2)d} \int_{\R^d} \mB_{j_1-s}(y_1) \sum_{j_2\leq j_1} \int_{\R^d} \mB_{j_2-s} (y_2) \chi_{\{|y_1-y_2|\lesssim 2^{j_1}\}}(y_2) \\
    &\quad\quad \sum_{j_3\leq j_2} \int_{\R^d} \mB_{j_3-s} (y_3) \chi_{\{|y_2-y_3|\lesssim 2^{j_1}\}}(y_3) \dif y_3 \dif y_2 \dif y_1 \\
    &\quad \lesssim 2^{-s\mu(d-1)} (\lambda C_{\Omega}^{-1})^2 \sum_{j\in\Z} \sum_{Q\in\mathfrak{Q}_{j-s}} \|b_Q\|_{L^1(\R^d)}.
    \end{align*}
Combining this with \eqref{Card Theta} and \eqref{sum|b_Q|_1} yields
    \begin{align*}
    \text{LHS of } (\ref{L_3 estimate for I_4})
    &\lesssim 2^{s\mu ( \lfloor d/2 \rfloor +1)} 2^{s\mu (d-1)} \|\Omega\|_{\infty} \Big(2^{-s\mu(d-1)} (\lambda C_{\Omega}^{-1})^2 3^{-i} \|f\|_{L^1(\R^d)} \Big)^{\frac{1}{3}} \\
    &= 2^{s\mu ( \lfloor d/2 \rfloor + 1) + \frac{2}{3}s\mu (d-1) } \|\Omega\|_{\infty} \Big( (\lambda C_{\Omega}^{-1})^2 3^{-i} \|f\|_{L^1(\R^d)} \Big)^{\frac{1}{3}},
    \end{align*}
which proves \eqref{L_3 estimate for I_4}.

\medskip
\noindent\textbf{Proof of the \(L^1\) estimate \eqref{L_1 estimate for I_4}.}
Let $\psi\in C^\infty$ be supported in $\{\xi\in\mathbb{R}^d: 1/2\le|\xi|\le2\}$ such that $\sum\limits_{k\in\mathbb{Z}} \psi_k (\xi)=1$ for $\xi\neq 0$, where $\psi_k (\xi) = \psi(2^k \xi)$.
Choose a radial function $\widetilde\psi\in C^\infty$ with $0\le\widetilde\psi\le1$, $\widetilde\psi(\xi)=1$ on $1/2\le|\xi|\le2$, and $\operatorname{supp}\widetilde\psi\subset\{1/4\le|\xi|\le4\}$. 
Set $\widetilde\psi_k(\xi)=\widetilde\psi(2^k\xi)$. 
Then 
    \[
    \psi_k = \widetilde\psi_k\psi_k.
    \]
Define the operators $\Lambda_k$ and $\widetilde\Lambda_k$ by
    \begin{align*}
    \Lambda_k h := \mF^{-1}(\psi_k) * h, \quad
    \widetilde{\Lambda}_k h := \mF^{-1}(\widetilde{\psi}_k) * h.   
    \end{align*}
We have 
    \begin{align} \label{Lambda_k=, I=}
    \Lambda_k = \widetilde{\Lambda}_k \Lambda_k, \quad
    I= \sum_{k\in\Z} \Lambda_k^2. 
    \end{align}
Applying the second indentity in (\ref{Lambda_k=, I=}), the left-hand side of (\ref{L_1 estimate for I_4}) becomes    
    \begin{align*}
    &\Big\| \sum_{j\in\Z} \sum_{Q\in\mathfrak{Q}_{j-s}} \sum_v \sum_{k\in\Z} \big(I - P_{j-s\kappa} \big) \big(I - G^s_v \big) \Lambda_k^2 \big( \epsilon_{J_Q}  \big( T_{j}^{s,v} b_{Q} \big)\widetilde{\chi}_{J_Q^m} \big)  \Big\|_{L^1(\R^d)} \\
    &\quad \leq  \sum_{j\in\Z} \sum_{Q\in\mathfrak{Q}_{j-s}} \sum_v \sum_{k\in\Z}  \Big\| \big(I - P_{j-s\kappa} \big) \big(I - G^s_v \big) \Lambda_k^2 \big( \epsilon_{J_Q}  \big( T_{j}^{s,v} b_{Q} \big)\widetilde{\chi}_{J_Q^m} \big)  \Big\|_{L^1(\R^d)} \\
     &\quad \leq \sum_{j\in\Z} \sum_{Q\in\mathfrak{Q}_{j-s}} \sum_v \Big( \sum_{k\leq j-s \sigma} +  \sum_{k> j-s \sigma} \Big) \\
     &\qquad \qquad \qquad \big\| (I - P_{j-s\kappa}) \Lambda_k  \big\|_{L^1\rightarrow L^1}   \Big\| \big(I - G^s_v \big) \Lambda_k \big( \epsilon_{J_Q}  \big( T_{j}^{s,v} b_{Q} \big)\widetilde{\chi}_{J_Q^m} \big)  \Big\|_{L^1(\R^d)}.
    \end{align*}
Here we have split the sum over \(k\) into low and high frequencies with respect to the scale \(j-s\sigma\), where \(\sigma\in(0,1)\) satisfies the parameter condition (\ref{0 < iota << mu << kappa << sigma < 1}).

We need the following three technical lemmas, whose proofs are deferred to the next subsection.
\begin{lemma} \label{||I - P_{j-s kappa} Lambda_k||}
Using the notation established above, for fixed $j$, $k\in\Z$, we have
	\begin{align*}
	\big\| (I - P_{j-s\kappa}) \Lambda_k  \big\|_{L^1\rightarrow L^1}
	\lesssim \min\{1, 2^{j-s\kappa-k}\}.
	\end{align*}
\end{lemma}

\begin{lemma} \label{sdhcjlah}
Using the notation established above, for fixed $j$, $k\in\Z$, $Q\in\mathfrak{Q}_{j-s}$, $e^s_v\in\Theta_s$, there exists an positive integer $N_1$, such that for every positive integer $N_0$,
    \begin{align*}
    &\big\| \big(I - G^s_v \big) \Lambda_k \big( \epsilon_{J_Q}  \big( T_{j}^{s,v} b_{Q} \big)\widetilde{\chi}_{J_Q^m} \big)  \big\|_{L^1(\R^d)} \\
    &\quad \lesssim C_{N_0} s^{2\gamma^{-1}N_0} 2^{-(j-k)N_0 + s\kappa N_0 + s\mu N_0 + 2s\mu N_1 + s\iota -s\iota(d-1)} \|\Omega\|_1 \|b_Q\|_{L^1(\R^d)}.
    \end{align*}
\end{lemma}

\begin{lemma} \label{kajsdhfska}
Using the notation established above, for fixed $j$, $k\in\Z$, $Q\in\mathfrak{Q}_{j-s}$, $e^s_v\in\Theta_s$, we have
    \begin{align*}
    &\big\| \big(I - G^s_v \big) \Lambda_k \big( \epsilon_{J_Q}  \big( T_{j}^{s,v} b_{Q} \big)\widetilde{\chi}_{J_Q^m} \big)  \big\|_{L^1(\R^d)} \\
    &\quad \lesssim 2^{-s\mu (d-1) + s\iota}(2^{j-s-k} + 2^{-s} s^{2\gamma^{-1}} + 2^{-s(1-\kappa)}) \|\Omega\|_1 \|b_Q\|_{L^1(\R^d)}.
    \end{align*}
\end{lemma}


Using Lemma~\ref{||I - P_{j-s kappa} Lambda_k||} (with the bound \(1\)) and Lemma~\ref{sdhcjlah}, together with \eqref{Card Theta} and \eqref{sum|b_Q|_1}, the low-frequency part is estimated as follows:
    \begin{align*}
	&\sum_{j\in\Z}  \sum_{Q\in\mathfrak{Q}_{j-s}} \sum_v \sum_{k\leq j-s\sigma} \big\| (I - P_{j-s\kappa}) \Lambda_k  \big\|_{L^1\rightarrow L^1}   \Big\| \big(I - G^s_v \big) \Lambda_k \big( \epsilon_{J_Q}  \big( T_{j}^{s,v} b_{Q} \big)\widetilde{\chi}_{J_Q^m} \big)  \Big\|_{L^1(\R^d)} \\
	&\quad \lesssim  C_{N_0} s^{2\gamma^{-1}N_0} 2^{-s\sigma N_0 + s\kappa N_0 + s\mu N_0 + 2s\mu N_1 + s\iota} \|\Omega\|_1 \sum_{j\in\Z} \sum_{Q\in\mathfrak{Q}_{j-s}} \|b_Q\|_{L^1(\R^d)} \\
	&\quad \lesssim  C_{N_0} s^{2\gamma^{-1}N_0} 2^{-s\sigma N_0 + s\kappa N_0 + s\mu N_0 + 2s\mu N_1 + s\iota} 3^{-i} \|\Omega\|_1 \|f\|_{L^1(\R^d)}.
    \end{align*}
This gives the term \(C_{N_0} s^{2\gamma^{-1}N_0} 2^{-\beta_1 s}\) in \eqref{L_1 estimate for I_4}, with
    \[
    \beta_1 
    := \sigma N_0 - (\kappa N_0+\mu N_0+2\mu N_1 + \iota).
    \]

For the high-frequency part, we use Lemma~\ref{||I - P_{j-s kappa} Lambda_k||} (with the bound $2^{j-s\kappa-k}$), Lemma~\ref{kajsdhfska}, (\ref{Card Theta}),  and (\ref{sum|b_Q|_1}). 
Summing over $k > j - s \sigma$ gives a geometric series:
    \begin{align*}
    &\sum_{j\in\Z} \sum_{Q\in\mathfrak{Q}_{j-s}} \sum_v \sum_{k > j-s \sigma} \big\| (I - P_{j-s\kappa}) \Lambda_k  \big\|_{L^1\rightarrow L^1}   \Big\| \big(I - G^s_v \big) \Lambda_k \big( \epsilon_{J_Q}  \big( T_{j}^{s,v} b_{Q} \big)\widetilde{\chi}_{J_Q^m} \big)  \Big\|_{L^1(\R^d)} \\
    &\quad \lesssim   \sum_{j\in\Z} \sum_{Q\in\mathfrak{Q}_{j-s}} \sum_{k> j-s \sigma} 2^{j-s\kappa-k} 2^{s\iota} (2^{j-s-k} + 2^{-s} s^{2\gamma^{-1}} + 2^{-s(1-\kappa)}) \|\Omega\|_1 \|b_Q\|_{L^1(\R^d)} \\
    &\quad \lesssim \big(2^{2s\sigma - s\kappa + s\iota - s} + 2^{s\sigma - s\kappa + s\iota} (2^{-s} s^{2\gamma^{-1}} + 2^{-s(1-\kappa)})\big) 3^{-i} \|\Omega\|_1 \|f\|_{L^1(\R^d)}.
    \end{align*}
This yields the terms
    \[
    2^{-\beta_2 s},\qquad s^{2\gamma^{-1}}2^{-\beta_3 s},\qquad 2^{-\beta_4 s}, 
    \]
with 
    \[
    \beta_2 := 1 - (2\sigma + \iota - \kappa),\qquad
    \beta_3 := 1 - (\sigma + \iota - \kappa),\qquad
    \beta_4 := 1 - (\sigma + \iota).
    \]

Combining the estimates for the low- and high-frequency parts gives exactly \eqref{L_1 estimate for I_4}, we complete the proof.

\medskip
\subsection{Proofs of Lemma \ref{||I - P_{j-s kappa} Lambda_k||}-\ref{kajsdhfska} } 
We now prove the three technical lemmas used in the preceding section.

\begin{proof}[Proof of Lemma \ref{||I - P_{j-s kappa} Lambda_k||}]
By the definitions of $P_{j-s\kappa}$ and $\Lambda_k$, we have
    \begin{align*}
    \big\| (I - P_{j-s\kappa}) \Lambda_k (h) \big\|_{L^1(\R^d)}
    &\leq \big\| \mF^{-1}(\psi_k) - \omega_{j-s\kappa}*\mF^{-1}(\psi_k) \big\|_{L^1(\R^d)} \| h \|_{L^1(\R^d)}.
    \end{align*}
On the one hand, 
    \begin{align*}
    &\big\| \mF^{-1}(\psi_k) - \omega_{j-s\kappa}*\mF^{-1}(\psi_k) \big\|_{L^1(\R^d)} \\
    &\quad \leq \big\| \mF^{-1}(\psi_k)	 \big\|_{L^1(\R^d)} +  \| \omega_{j-s\kappa} \|_{L^1(\R^d)} \big\|\mF^{-1}(\psi_k) \big\|_{L^1(\R^d)} \\
    &\quad \leq \infty.
    \end{align*}
On the other hand, using the mean value theorem, $\| \nabla \mathcal{F}^{-1}(\psi_k) \|_{L^1(\R^d)} \lesssim 2^{-k}$ and $|y|\le 2^{j-s\kappa}$, we obtain
    \begin{align*}
    \int_{\R^d} \big| \mF^{-1}(\psi_k)(x) - \mF^{-1}(\psi_k)(x-y) \big| \dif x
    &\leq \int_0^1 \int_{\R^d} \big| \nabla \big(\mF^{-1}(\psi_k) \big)(x-ty) \big| \dif x \dif t \cdot|y| \\
    &\lesssim 2^{j-s\kappa} \big\| \mF^{-1}(\psi_k) \big\|_{L^1(\R^d)} \\
    &\lesssim  2^{j-s\kappa -k}.
    \end{align*}
Therefore,
    \begin{align*}
    &\big\| \mF^{-1}(\psi_k) - \omega_{j-s\kappa}*\mF^{-1}(\psi_k) \big\|_{L^1(\R^d)}  \\
    &\leq \int_{\R^d} \omega_{j-s\kappa}(y) \int_{\R^d} \big| \mF^{-1}(\psi_k)(x) - \mF^{-1}(\psi_k)(x-y) \big| \dif x \dif y \\
    &\lesssim 2^{j-s\kappa -k}.
    \end{align*}
Combining the two estimates gives
    \[
    \| (I-P_{j-s\kappa})\Lambda_k \|_{L^1\to L^1} \lesssim \min\{1,\; 2^{j-s\kappa - k}\},
    \]
which completes the proof.
\end{proof}

\medskip
\begin{proof}[Proof of Lemma \ref{sdhcjlah}]
By the definitions of $G^s_v$ and $\Lambda_k$,     
    \begin{align*}
    &\big(I - G^s_v \big) \Lambda_k \big( \epsilon_{J_Q}  \big( T_{j}^{s,v} b_{Q} \big) \widetilde{\chi}_{J_Q^m} \big) (x) \\
    &\quad= (2\pi)^{-\frac{d}{2}} \int_{\R^d} e^{i \langle x,\xi \rangle} \big( 1 - \Phi ( 2^{s\mu} \rangle e^s_v,\ \xi/|\xi| \rangle) \big) \psi_k(\xi) \mF\big( \epsilon_Q \big(T^{s,v}_j b_Q\big) \widetilde{\chi}_{J_Q^m} \big) \dif\xi.
    \end{align*}
Note that
    \begin{align*}
    &\mF\big( \epsilon_Q \big(T^{s,v}_j b_Q\big) \widetilde{\chi}_{J_Q^m} \big) (t) \\
    &\quad = (2\pi)^{-\frac{d}{2}} \epsilon_{J_Q} \int_{\R^d} \int_{\R^d} e^{-i \langle \xi,t \rangle } \Omega(t-y) K_j^s(t,y) \chi_{E^s_v}\Big( \frac{t-y}{|t-y|} \Big) b_Q(y) \widetilde{\chi}_{J_Q^m}(t) \dif y\dif t,
    \end{align*}       
we obtain   
    \begin{align*}
    \big(I - G^s_v \big) \Lambda_k \big( \epsilon_{J_Q}  \big( T_{j}^{s,v} b_{Q} \big) \widetilde{\chi}_{J_Q^m} \big) (x)
    = \epsilon_{J_Q} \int_Q D^{s,v}_{j,k} (x,y) b_Q(y) \dif y, 
    \end{align*}   
where the kernel \(D_{j,k}^{s,v}(x,y)\) is given by
    \begin{align*}
    D^{s,v}_{j,k} (x,y)
    = (2\pi)^{-d} \int_{\R^d} \int_{\R^d} & e^{i \langle x,\xi \rangle } \big( 1 - \Phi ( 2^{s\mu} \langle e^s_v,\ \xi/|\xi| \rangle ) \big) e^{-i \rangle \xi,t \rangle } \Omega(t-y)  \\
    & K_j^s(t,y) \chi_{E^s_v}\Big( \frac{t-y}{|t-y|} \Big) \widetilde{\chi}_{J_Q^m}(t) \psi_k(\xi) \dif t \dif\xi.
    \end{align*}    
From this representation, it follows immediately that
    \begin{align*}
    \big\| \big(I - G^s_v \big) \Lambda_k \big( \epsilon_{J_Q}  \big( T_{j}^{s,v} b_{Q} \big) \widetilde{\chi}_{J_Q^m} \big)  \big\|_{L^1(\R^d)}
    \leq \sup_{y\in Q} \|D^{s,v}_{j,k}(\cdot, y) \|_{L^1(\R^d)} \|b_Q\|_{L^1(\R^d)}.
    \end{align*}
So we are reduced to estimating the \(L^1\) norm of \(D_{j,k}^{s,v}(\cdot,y)\) uniformly in \(y\in Q\).

To proceed, set
    \begin{align*}
    h_{k,s,v} (\xi) 
    = \big( 1 - \Phi ( 2^{s\mu} \langle e^s_v,\ \xi/|\xi| \rangle ) \big) \psi_k(\xi).
    \end{align*}    
Making a change of variables $t-y=r\theta$, we rewrite $D^{s,v}_{j,k} (x,y)$ as
    \begin{equation}
    \begin{aligned} \label{D^{s,v}_{j,k} (x,y)}
    D^{s,v}_{j,k} (x,y)
    = (2\pi)^{-d}  &\int_{E^s_v} \Omega(\theta) \int_{\R^d} e^{i \langle x-y, \xi \rangle } h_{k,s,v}(\xi)  \\
    &\cdot \Big( \int_0^{+\infty}	 e^{-ir \langle \xi, \theta \rangle } K_j^s(r\theta+y,y) \widetilde{\chi}_{J_Q^m} (r\theta+y) r^{d-1} \dif r \Big) \dif\xi \dif\theta.
    \end{aligned}
    \end{equation}   
For the integral involving $r$, we integrate by parts $N_0$ times with respect to $r$.
Hence, this integral can be represented as  
    \begin{align*}
    (i \langle \xi, \theta \rangle )^{-N_0} \int_0^{+\infty} e^{-ir \langle \xi, \theta \rangle } \partial^{N_0}_r \big( K_j^s(r\theta+y,y) \widetilde{\chi}_{J_Q^m} (r\theta+y) r^{d-1} \big) \dif r.
    \end{align*}
Substituting this back into \eqref{D^{s,v}_{j,k} (x,y)} and then integrating by parts \(N_1\) times with respect to \(\xi\) (where \(N_1\) will be fixed below), the kernel \(D_{j,k}^{s,v}(x,y)\) equals
    \begin{equation}
    \begin{aligned} \label{D^{s,v}_{j,k} (x,y)-2}
    &(2\pi)^{-d} \int_{E^s_v} \Omega(\theta) \int_{ 2^{j-5} }^{ 2^{j-1} } \partial^{N_0}_r \big( K_j^s(r\theta+y,y) \widetilde{\chi}_{J_Q^m} (r\theta+y) r^{d-1} \big) \\
    &\qquad \int_{|\xi|\approx 2^{-k}} \frac{(I-2^{-2k} \Delta_{\xi})^{N_1} }{ \big(1+2^{-k} |x-y-r\theta|^2 \big)^{N_1} } \big( h_{k,s,v}(\xi) ( \langle \xi, \theta \rangle )^{-N_0} \big) e^{i \langle x-y-r\theta, \xi \rangle } \dif\xi \dif r \dif\theta,
    \end{aligned}
    \end{equation}
where the ranges of $r$ and $\xi$ follow from (\ref{estimate for K_j^s(x,y)}) and $\xi \in \supp\psi_k$, respectively.

To handle (\ref{D^{s,v}_{j,k} (x,y)-2}), we need the following three estimates: 
     \begin{align} \label{qwer}
     \big| \partial^{N_0}_r \big( K_j^s(r\theta+y,y) \widetilde{\chi}_{J_Q^m} (r\theta+y) r^{d-1} \big) \big| 
     \lesssim s^{2N_0 \gamma^{-1}} 2^{-(N_0+1)j + s\kappa N_0},	
     \end{align}
     
     \begin{align} \label{zxcv}
 	 \int_{|\xi|\approx 2^{-k}} \int_{\R^d} \big(1+2^{-k} |x-y-r\theta|^2 \big)^{-N_1} \dif x \dif \xi \lesssim 1, 
     \end{align} 
     
     \begin{align} \label{asdf}
 	 \big| (I-2^{-2k} \Delta_{\xi})^{N_1} \big( h_{k,s,v}(\xi) (i\langle \xi, \theta\rangle )^{-N_0} \big) \big|
 	 \lesssim C_{N_0} 2^{(k+s\mu)N_0 + 2s\mu N_1},
     \end{align}      
where \(N_1\) will be chosen in the proof of \eqref{zxcv}.

We first prove (\ref{qwer}). 
For \(2^{j-5}\le r\le 2^{j-1}\), we have 
    \[
    \big|\partial_r^i \big( K_j^s(r\theta+y,y) r^{d-1} \big) \big| \lesssim C_i 2^{-j(i+1)} s^{2i \gamma^{-1}},
    \]
and 
    \[
    \big| \partial_r^i \big( \widetilde{\chi}_{J_Q^m} (r\theta+y) \big) \big| \lesssim 2^{-i(j-s\kappa)}.
    \]
for $i \geq 1$.
Then, applying the product rule,
    \begin{align*}
    &\big|\partial^{N_0}_r \big( K_j^s(r\theta+y,y) \widetilde{\chi}_{J_Q^m} (r\theta+y) r^{d-1} \big) \big| \\
    &\quad \lesssim \sum_{i=0}^{N_0} \binom{N_0}{i} \big| \partial_r^i \big( K_j^s(r\theta+y,y) \big) \big| \big| \partial_r^{N_0-i} \big( \widetilde{\chi}_{J_Q^m} (r\theta+y) \big) \big| \\
    &\quad \lesssim C_{N_0} 2^{-(N_0+1)j + s\kappa N_0} s^{2N_0 \gamma^{-1}},
    \end{align*}
which proves (\ref{qwer}).

Next, we establish (\ref{zxcv}). 
Take $N_1 := \lfloor d/2 \rfloor +1$.
Consequently, it is easy to see
    \begin{align*}
    \int_{|\xi|\approx 2^{-k}} \int_{\R^d} \big(1+2^{-k} |x-y-r\theta|^2 \big)^{-N_1} \dif x \dif \xi
    \lesssim 2^{kd} \int_{|\xi|  \approx 2^{-k} } \dif\xi
    \lesssim 1,
    \end{align*}
so \eqref{zxcv} holds.

It remains to show (\ref{asdf}). 
In fact, we prove the more general estimate:
    \begin{align*} 
     \big| (I-2^{-2k} \Delta_{\xi})^{N} \big( h_{k,s,v}(\xi) (i\langle \xi, \theta\rangle )^{-N_0} \big) \big|
 	 \lesssim C_{N_0} 2^{(k+s\mu)N_0 + 2s\mu N},
     \end{align*}  
for every $N\geq 0$.
Consider the case $N = 0$.
Since $\theta\in E^s_v$, we get $|e^s_v - \theta| \leq 2^{-s\mu-4}$.
Furthermore, when $h_{k,s,v}(\xi)\neq 0$, we have $\Phi ( 2^{s\mu} \langle e^s_v,\ \xi/|\xi| \rangle ) \neq 1$, which implies $| \langle e^s_v,\ \xi/|\xi| \rangle | \gtrsim 2^{2-s\mu}$.
Combining these results with $|\xi|\approx 2^{-k}$ (since $\xi \in \supp\psi_k$), we see 
    \begin{align*}
    |\langle \xi,\theta \rangle|
    \geq \big( |\langle \xi/|\xi|, e^s_v \rangle | - |\langle\xi/|\xi|, e^s_v - \theta \rangle| \big) \cdot |\xi| 
    \gtrsim 2^{-k-s\mu}. 
    \end{align*} 
Therefore,
    \begin{align*}
    | h_{k,s,v}(\xi) (i \langle \xi, \theta \rangle)^{-N_0} | 
    \lesssim |\langle \xi, \theta \rangle|^{-N_0}
    \lesssim 2^{(k+s\mu)N_0}.	
    \end{align*}
For $N = 1$, expanding $\Delta_{\xi} \big( h_{k,s,v}(\xi) (i \langle \xi, \theta \rangle )^{-N_0} \big) $ as
    \begin{align*}
    (i \langle \xi, \theta \rangle )^{-N_0} \Delta_{\xi} h_{k,s,v} (\xi) 
    + 2 \nabla h_{k,s,v} (\xi) \cdot \nabla (i \langle \cdot, \theta \rangle )^{-N_0} (\xi) 
    + h_{k,s,v}(\xi) \Delta_{\xi} (i \langle \xi, \theta \rangle )^{-N_0} .
    \end{align*}
Using the standard derivative bounds: 
for any multi-index \(\alpha\),
    \[
    |\partial_\xi^\alpha h_{k,s,v}| \lesssim 2^{(s\mu+k)|\alpha|}, 
    \quad
    \big|\partial_\xi^\alpha \big(\langle \xi,\theta\rangle^{-N_0} \big) \big| \lesssim C_{N_0} 2^{(k+s\mu)(N_0+|\alpha|)},
    \]
we obtain
    \[
    |\Delta_{\xi} \big( h_{k,s,v}(\xi) (i \langle \xi, \theta \rangle )^{-N_0} | 
    \lesssim C_{N_0} 2^{(k+s\mu)(N_0+2)}.
    \]
Consequently,
    \[
    \big|(I-2^{-2k}\Delta_\xi) \big( h_{k,s,v}(\xi) (i \langle \xi, \theta \rangle)^{-N_0} \big) \big|
    \lesssim C_{N_0} 2^{(k+s\mu)N_0 + 2s\mu},
    \]
which proves the case \(N=1\).
The case \(N\ge2\) follows by induction on \(N\). Hence \eqref{asdf} holds for all \(N\ge0\), in particular for \(N=N_1\).

Finally, applying (\ref{qwer}), \eqref{zxcv}, (\ref{asdf}), and the condition $\|\Omega\|_{\infty} \leq 2^{s\iota} \|\Omega\|_{1}$ to \eqref{D^{s,v}_{j,k} (x,y)-2}, we have
    \begin{align*}
    (\ref{D^{s,v}_{j,k} (x,y)-2})
    &\lesssim s^{2N_0 \gamma^{-1}}  2^{-(N_0+1)j + s\kappa N_0} C_{N_0} 2^{(k+s\mu)N_0 + 2s\mu N_1} \int_{j-5}^{j-1} \int_{E^s_v} |\Omega(\theta)| \dif\theta \dif r \\
    &\lesssim C_{N_0} s^{2N_0 \gamma^{-1}} 2^{-(j-k)N_0 + s\kappa N_0 + s\mu N_0 + 2s\mu N_1 + s\iota - s\mu (d-1)} \|\Omega\|_1.
    \end{align*}  
This is exactly the bound required for \(\|D_{j,k}^{s,v}(\cdot,y)\|_{L^1(\R^d)}\), and hence completes the proof of the lemma.
 
\end{proof}

\medskip
\begin{proof}[Proof of Lemma \ref{kajsdhfska}]
First, we recall the standard estimate
    \begin{align*}
    \big\| \big(I - G^s_v \big) \widetilde{\Lambda}_k   \big\|_{L^1\rightarrow L^1}
    \lesssim 1,	
    \end{align*}
which holds uniformly in \(s\), \(e_v^s\), and \(k\) (see \cite{Seeger1996Singular}).    
Combining this with \(\Lambda_k = \widetilde{\Lambda}_k \Lambda_k\) (from \eqref{Lambda_k=, I=}), we obtain    
    \begin{align*}
    \big\| \big(I - G^s_v \big)  \Lambda_k \big( \epsilon_{J_Q}  \big( T_{j}^{s,v} b_{Q} \big)\widetilde{\chi}_{J_Q^m} \big)  \big\|_{L^1(\R^d)}
    \lesssim \big\| \Lambda_k \big( \epsilon_{J_Q}  \big( T_{j}^{s,v} b_{Q} \big)\widetilde{\chi}_{J_Q^m} \big)  \big\|_{L^1(\R^d)}.
    \end{align*}
By the definitions of $\Lambda_k$ and $T_j^{s,v}$, 
    \begin{align*}
    \Lambda_k \big( \epsilon_{J_Q}  \big( T_{j}^{s,v} b_{Q} \big) \widetilde{\chi}_{J_Q^m} \big)  (x)
    = \int_{\R^d} A^{s,v}_{j,k} (x,y) b_Q(y) \dif y,
    \end{align*}
where
    \begin{align*}
    A^{s,v}_{j,k} (x,y)
    = (2\pi)^{-d} \int_{\R^d} e^{i \langle x,\xi \rangle } \psi_k(\xi) \int_{\R^d} e^{-i \langle \xi,t \rangle} \Omega(t-y) K_j^s(t,y) \chi_{E^s_v}\big( \frac{t-y}{|t-y|} \big) \widetilde{\chi}_{J_Q^m}(t) \dif t \dif \xi.
    \end{align*}	
Using polar coordinates $t-y = r\theta$ and the support condition $2^{j-5}\le r\le 2^{j-1}$ (from \eqref{estimate for K_j^s(x,y)}), we rewrite $A_{j,k}^{s,v}$ as
    \begin{align*}
    &(2\pi)^{-d} \int_{E^s_v} \Omega(\theta) \int_{\R^d} \int_{2^{j-5}}^{2^{j-1}} e^{i \langle x-y-r\theta,\xi \rangle } \psi_k(\xi) K_j^s(y+r\theta,y) \widetilde{\chi}_{J_Q^m}(y+r\theta) r^{d-1} \dif r \dif \xi \dif\theta.
    \end{align*}
Integrating by parts $N_1 = \lfloor d/2 \rfloor +1$ times with respect to $\xi$ gives    
    \begin{align*}
    A^{s,v}_{j,k} (x,y)
    &= (2\pi)^{-d} \int_{E^s_v} \Omega(\theta) \int_{\R^d} \int_{2^{j-5}}^{2^{j-1}} e^{i \langle x-y-r\theta,\xi \rangle } \frac{(I-2^{-2k} \Delta_{\xi})^{N_1} \psi_k(\xi)}{\big(1+2^{-2k} |x-y-r\theta|^2 \big)^{N_1} }  \\ 
    & \quad \quad \quad \quad \quad \quad \quad \quad \quad \quad \quad \quad \quad \cdot K_j^s(y+r\theta,y)\widetilde{\chi}_{J_Q^m}(y+r\theta) r^{d-1} \dif r \dif \xi \dif\theta.
    \end{align*}    
Then, using this representation of $A^{s,v}_{j,k}$ and the cancellation condition of \(b_Q\) in (cz-ii), we have
    \begin{align*}
    \Lambda_k \big( \epsilon_{J_Q}  \big( T_{j}^{s,v} b_{Q} \big) \widetilde{\chi}_{J_Q^m} \big)  (x)
    &= \int_{\R^d} \big( A^{s,v}_{j,k} (x,y) - A^{s,v}_{j,k} (x,c(Q)) \big) b_Q(y) \dif y.
    \end{align*}
We decompose the difference $A^{s,v}_{j,k} (x,y) - A^{s,v}_{j,k} (x,c(Q)$ into four parts:
    \begin{align*}
    A^{s,v}_{j,k,1}(x,y)
    &= (2\pi)^{-d} \int_{E^s_v} \Omega(\theta) 
    \int_{\R^d} \int_{2^{j-5}}^{2^{j-1}} \big( e^{i\langle x-y-r\theta,\xi \rangle } - e^{i \langle x-c(Q)-r\theta,\xi \rangle } \big)  \\ 
    &\quad \frac{(I-2^{-2k} \Delta_{\xi})^{N_1} \psi_k(\xi)}{\big(1+2^{-2k} |x-y-r\theta|^2 \big)^{N_1} } K_j^s(y+r\theta,y) \widetilde{\chi}_{J_Q^m}(y+r\theta) r^{d-1} \dif r \dif \xi \dif\theta, \\
    A^{s,v}_{j,k,2}(x,y)
    &= (2\pi)^{-d} \int_{E^s_v} \Omega(\theta) 
    \int_{\R^d} \int_{2^{j-5}}^{2^{j-1}} e^{i \langle x-c(Q)-r\theta,\xi \rangle }  \\ 
    &\quad\Big( \frac{(I-2^{-2k} \Delta_{\xi})^{N_1} \psi_k(\xi)}{\big(1+2^{-2k} |x-y-r\theta|^2 \big)^{N_1} } - \frac{(I-2^{-2k} \Delta_{\xi})^{N_1} \psi_k(\xi)}{\big(1+2^{-2k} |x-c(Q)-r\theta|^2 \big)^{N_1} } \Big)\\ 
    &\quad K_j^s(y+r\theta,y) \widetilde{\chi}_{J_Q^m}(y+r\theta) r^{d-1} \dif r \dif \xi \dif\theta, \\	
    A^{s,v}_{j,k,3}(x,y)
    &= (2\pi)^{-d} \int_{E^s_v} \Omega(\theta) 
    \int_{\R^d} \int_{2^{j-5}}^{2^{j-1}} e^{i \langle x-c(Q)-r\theta,\xi \rangle } \frac{(I-2^{-2k} \Delta_{\xi})^{N_1} \psi_k(\xi)}{\big(1+2^{-2k} |x-y-r\theta|^2 \big)^{N_1}} \\ 
    &\quad \big( K_j^s(y+r\theta,y) -  K_j^s(c(Q)+r\theta,c(Q)) \big)\widetilde{\chi}_{J_Q^m}(y+r\theta) r^{d-1} \dif r \dif \xi \dif\theta, \\
    A^{s,v}_{j,k,4}(x,y)
    &= (2\pi)^{-d} \int_{E^s_v} \Omega(\theta) 
    \int_{\R^d} \int_{2^{j-5}}^{2^{j-1}} e^{i \langle x-c(Q)-r\theta,\xi \rangle } \frac{(I-2^{-2k} \Delta_{\xi})^{N_1} \psi_k(\xi)}{\big(1+2^{-2k} |x-y-r\theta|^2 \big)^{N_1} } \\ 
    &\quad K_j^s(c(Q)+r\theta,c(Q)) \big( \widetilde{\chi}_{J_Q^m}(y+r\theta) - \widetilde{\chi}_{J_Q^m}(c(Q)+r\theta) \big) r^{d-1} \dif r \dif \xi \dif\theta.
    \end{align*}
Consequently,
    \begin{align*}
    \big\| \Lambda_k \big( \epsilon_{J_Q}  \big( T_{j}^{s,v} b_{Q} \big) \widetilde{\chi}_{J_Q^m} \big) \big\|_{L^1(\R^d)}
    &\leq \sup_{y\in Q}  \big\| A^{s,v}_{j,k} (\cdot,y) - A^{s,v}_{j,k} (\cdot,c(Q))  \big\|_{L^1(\R^d)} \|b_Q\|_{L^1(\R^d)} \\
    &\leq  \sup_{y\in Q} \sum_{i=1}^4 \big\| A^{s,v}_{j,k,i} (\cdot,y) \big\|_{L^1(\R^d)} \|b_Q\|_{L^1(\R^d)}.
    \end{align*}
We now estimate each $A_{j,k,i}^{s,v}$ for $i$=1, 2, 3, 4.

\medskip 
\noindent\textbf{Estimate for $A_{j,k,1}^{s,v}$.}
It is straightforward to verify that
    \begin{align} \label{(I-2^{-2k}bigtriangleup)^N psi_k}
    \big| (I-2^{-2k} \Delta_{\xi})^{N_1} \psi_k(\xi) \big| \lesssim 1,
    \end{align}
and for $y \in Q$,
    \begin{align*}
    |e^{i \langle -y, \xi \rangle } - e^{i \langle -c(Q), \xi \rangle } | \lesssim 2^{j-s-\kappa}. 
    \end{align*}
Also, by \eqref{estimate for K_j^s(x,y)},
    \begin{align} \label{|K^s_j(y+r theta,y)| r^{d-1}}
    |K^s_j(y+r\theta,y) r^{d-1} |  
    \lesssim 2^{-j}.
    \end{align}
Furthermore, exactly as in the proof of \eqref{zxcv}, 
    \begin{align} \label{bnbnm}
 	\int_{|\xi|\approx 2^{-k}} \int_{\R^d} \big(1+2^{-2k} |x-y-r\theta|^2 \big)^{-N_1} \dif x \dif \xi 
 	\lesssim 1.
    \end{align}
Combining these estimates with $\|\Omega\|_{\infty}\le 2^{s\iota}\|\Omega\|_{1}$ and \eqref{|E^{s,v}_j|}, we get    
    \begin{align*}
    \int_{\R^d} \big| A^{s,v}_{j,k,1}(x,y) \big| \dif x
    \lesssim 2^{s\iota} 2^{j-s-\kappa} \|\Omega\|_1 |E^s_v|
    \lesssim 2^{-s\mu(d-1) + s\iota}  2^{j-s-\kappa} \|\Omega\|_1.
    \end{align*}

\medskip
\noindent\textbf{Estimate for $A_{j,k,2}^{s,v}$.}
For convenience, we set
    \begin{align*}
    V(y)= (1 + 2^{-2k} |x-y-r\theta|^2 )^{-N_1}. 
    \end{align*}
Therefore,  
    \begin{align*}
    |(\nabla V)(y)| 
    \lesssim 2^{-k} N_1 \frac{ 2^{-k}  |x-y-r\theta| }{ \big(1+2^{ -2k } |x-y-r\theta|^2 \big)^{N_1 + 1} }.
    \end{align*}    
By the mean value theorem,
    \begin{align*}
    |V(y) - V(c(Q))|
    &\leq \int_0^1 \big| \langle y-c(Q), (\nabla V) (ty + (1-t) c(Q) ) \rangle \big| \dif t \\
    &\lesssim 2^{j-s-k} N \int_0^1 \frac{ 2^{-k}  |x- (ty+(1-t)c(Q)) -r\theta| }{ \big(1+2^{-2k} |x- (ty+(1-t)c(Q)) -r\theta|^2 \big)^{N_1 + 1} } \dif t.
    \end{align*}
Meanwhile, a direct computation gives
    \begin{align*}
    \int_{|\xi|\approx 2^{-k}} \int_{\R^d} \frac{ 2^{-k}  |x- (ty+(1-t)c(Q)) -r\theta| }{ \big(1+2^{-2k} |x- (ty+(1-t)c(Q)) -r\theta|^2 \big)^{N_1 + 1} } \dif x \dif\xi
    \lesssim 1.
    \end{align*}
Finally, using (\ref{(I-2^{-2k}bigtriangleup)^N psi_k}), (\ref{|K^s_j(y+r theta,y)| r^{d-1}}), the above estimates, $\|\Omega\|_{\infty} \leq 2^{s\iota} \|\Omega\|_{1}$, and (\ref{|E^{s,v}_j|}), we obtain
    \begin{align*}
    \int_{\R^d} \big| A^{s,v}_{j,k,2}(x,y) \big| \dif  x 
    &\lesssim \int_{\R^d}  \int_{E^s_v}  |\Omega(\theta)| 
    \int_{\R^d} \int_{2^{j-5}}^{2^{j-1}}  \big|(I-2^{-2k} \Delta_{\xi})^{N_1} \psi_k(\xi) \big| \big| V(y)-V(c(Q)) \big| \\ 
    &\quad \quad \quad \big| K_j^s(y+r\theta,y) r^{d-1} \big| |\widetilde{\chi}_{J_Q^m}(y+r\theta)|  \dif r \dif \xi \dif\theta \\
    &\lesssim 2^{j-s-\kappa} 2^{ s\iota - s\mu (d-1) } \|\Omega\|_1.
    \end{align*}

\medskip
\noindent\textbf{Estimate for $A_{j,k,3}^{s,v}$.}
We first bound the difference of the kernels. 
Split $ \big| ( K_j^s(y+r\theta,y) -  K_j^s(c(Q)+r\theta,c(Q)) \big) \big| $ into
    \begin{align*}
    \big|K_j^s(y+r\theta,y) -  K_j^s(c(Q)+r\theta,y)\big| 
    + \big| K_j^s(c(Q)+r\theta,y) -  K_j^s(c(Q)+r\theta,c(Q)) \big|.
    \end{align*} 
Both terms are handled in the same way; we only treat the first.

For a fixed cube $Q\in\mathfrak{Q}_{j-s}$ and $y\in Q$, the mean value theorem gives
    \begin{align*}
    &|\eta_{j-s_{\gamma}}(y + r\theta - u) - \eta_{j-s_{\gamma}}(c(Q) + r\theta - u) | \\
    &\quad \leq \int_0^1 \big| \langle y-c(Q), (\nabla\eta_{j-s_{\gamma} \rangle) (ty + (1-t)c(Q) + r\theta -u) } \big| \dif t \\
    &\quad \lesssim 2^{j-s} \int_0^1 \big| (\nabla\eta_{j-s_{\gamma}}) (ty + (1-t)c(Q) + r\theta -u)  \big| \dif t.
    \end{align*}
Using the definition of \(K_j^s\) as a convolution of \(\eta_{j-s_\gamma}\) with \(K_j\), together with \eqref{|K_j(u,v)|} and \eqref{2^{s_{gamma}}}, we get
    \begin{align*}
    &\big|K_j^s(y+r\theta,y) -  K_j^s(c(Q)+r\theta,y)\big| \\
    &\quad \lesssim 2^{j-s-jd} \Big( \|\nabla \eta_{j-s_{\gamma}}\|_{\infty} \int_{|ty + (1-t)c(Q) + r\theta -u| \leq 2^{j-s_{\gamma}}} \dif u \Big)  \int_0^1 \int_{\R^d} |\eta_{j-s_{\gamma}} (y-v)| \dif v \dif t \\
    &\quad \lesssim s^{2 \gamma^{-1}} 2^{-jd-s}.
    \end{align*}

Now, applying this kernel difference estimate, $\|\Omega\|_{\infty} \leq 2^{s\iota} \|\Omega\|_{1}$, \eqref{(I-2^{-2k}bigtriangleup)^N psi_k}, \eqref{bnbnm} (with $c(Q)$ in place of $y$), and (\ref{|E^{s,v}_j|}), we have
    \begin{align*}
    &\int_{\R^d} \big| A^{s,v}_{j,k,3}(x,y) \big| \dif  x \\
    &\quad \lesssim \int_{\R^d} \int_{E^s_v} |\Omega(\theta)| 
    \int_{\R^d} \int_{2^{j-5}}^{2^{j-1}} \frac{ \big|(I-2^{-2k} \Delta_{\xi})^{N_1} \psi_k(\xi) \big| }{\big(1+2^{-2k} |x-y-r\theta|^2 \big)^{N_1} } \\
    &\quad \quad \quad \quad \big| K_j^s(y+r\theta,y) -  K_j^s(c(Q)+r\theta,c(Q)) \big| |\widetilde{\chi}_{J_Q^m}(y+r\theta)| r^{d-1} \dif r \dif \xi \dif\theta \dif x \\
    &\quad \lesssim 2^{s\iota - s\mu (d-1) -s} s^{2\gamma^{-1}} \|\Omega\|_1. 
    \end{align*}


\medskip
\noindent\textbf{Estimate for $A_{j,k,4}^{s,v}$.}
For a fixed cube $Q\in\mathfrak{Q}_{j-s}$ and $y\in Q$, the mean value theorem and $\| \nabla \widetilde{\chi}_{J_Q^m} \|_{\infty} = 2^{-j+s\kappa} \| \nabla \omega \| \lesssim 2^{-j+s\kappa}$ yield
    \begin{align*}
    \big| \widetilde{\chi}_{J_Q^m}(y+r\theta) - \widetilde{\chi}_{J_Q^m} \big(c(Q)+r\theta \big) \big|
    &\leq \int_0^1 \big\langle y-c(Q), \big(\nabla \widetilde{\chi}_{J_Q^m} \big) \big( ty + (1-t)c(Q) + r\theta \big) \big\rangle \dif t	\\
    &\lesssim 2^{j-s} \| \nabla \widetilde{\chi}_{J_Q^m} \|_{\infty} \\
    &\lesssim 2^{-s + s\kappa}.
    \end{align*}
Combining this with $\|\Omega\|_{\infty} \leq 2^{s\iota} \|\Omega\|_{1}$, \eqref{(I-2^{-2k}bigtriangleup)^N psi_k}, \eqref{|K^s_j(y+r theta,y)| r^{d-1}}, \eqref{bnbnm} (with \(c(Q)\) in place of \(y\)), and \eqref{|E^{s,v}_j|}, we see
    \begin{align*}
    &\int_{\R^d} \big| A^{s,v}_{j,k,4}(x,y) \big| \dif  x \\
    &\quad \lesssim \int_{\R^d} \int_{E^s_v} |\Omega(\theta)| 
    \int_{\R^d} \int_{2^{j-5}}^{2^{j-1}}  \frac{ \big|(I-2^{-2k} \Delta_{\xi})^{N_1} \psi_k(\xi) \big| }{\big(1+2^{-2k} |x-y-r\theta|^2 \big)^{N_1} } \big| K_j^s(c(Q)+r\theta,c(Q)) \big|\\ 
    &\quad  \quad \quad \quad \quad  \big| \widetilde{\chi}_{J_Q^m}(y+r\theta) - \widetilde{\chi}_{J_Q^m}(c(Q)+r\theta) \big| r^{d-1} \dif r \dif \xi \dif\theta \dif x \\
    &\quad \lesssim 2^{s\iota -s\mu(d-1)} 2^{-s+s\kappa} \|\Omega\|_1.
    \end{align*}

\medskip    
Collecting the estimates for \(A_{j,k,1}^{s,v}\) through \(A_{j,k,4}^{s,v}\), we arrive at
    \[
    \begin{aligned}
    &\| (I - G_v^s) \Lambda_k \bigl( \epsilon_{J_Q} (T_j^{s,v} b_Q) \widetilde{\chi}_{J_Q^m} \bigr) \|_{L^1(\R^d)} \\
    &\quad \lesssim 2^{-s\mu(d-1)+s\iota} \bigl( 2^{j-s-k} + 2^{-s}s^{2\gamma^{-1}} + 2^{-s(1-\kappa)} \bigr) \|\Omega\|_1 \|b_Q\|_{L^1(\R^d)}.
    \end{aligned}
    \]
This completes the proof of Lemma \ref{kajsdhfska}.    		
\end{proof}

\medskip
\section{applications of the main theorem} \label{Section: applications of the main theorem}
In this section, we prove the weak type $(1,1)$ estimates for the three maximal operators
introduced in the introduction. We refer to their definitions there and
focus on the kernel estimates needed to apply Theorem~\ref{main theorem}.
Throughout this section, $d\geq2$ and $C_\Omega$ is defined by \eqref{C_Omega}.

\medskip
\subsection{Maximal Calder\'on commutator} 
We first apply Theorem~\ref{main theorem} to the maximal Calder\'on
commutator. 
The Lipschitz condition on $A$ gives the required kernel estimates and leads to the following weak type $(1,1)$ bound.

    \begin{theorem} 
    Suppose $\Omega\in L \log L(\sS^{d-1})$ satisfies (\ref{homogeneous function of degree zero}), (\ref{cancellation condition of order 1}).
    Let $A$ be a Lipschitz function on $\R^d$.
    Then, the maximal Calder\'on commutator $T_{\Omega,A,*}$ defined as in (\ref{maximal Calderon commutator}) is of weak type $(1,1)$. 
    That is, for any $f\in L^1(\R^d)$ and $\lambda>0$, 
 	\[
 	\lambda \big| \{x\in\R^d: T_{\Omega, A, *}f(x)>\lambda \} \big| 
 	\lesssim C_{\Omega} \|\nabla A\|_{L^{\infty}(\R^d)} \|f\|_{L^1(\R^d)}.
 	\]
    \end{theorem}

\begin{proof}
	Set
    \[
    K_A(x,y)=\frac{A(x)-A(y)}{|x-y|^{d+1}}.
    \]
By the $L^2$ boundedness of $T_{\Omega,A,*}$ (see \cite{Bajvsanski1967On}) and Theorem \ref{main theorem}, it is enough to show that the kernel $K(x,y)$ satisfies (\ref{Size condition}) and (\ref{regularity condition}).

First, since $A$ is a Lipschitz function on $\R^d$, the size condition (\ref{Size condition}) is easy to check.

Next, for fixed $y\in \R^d$ and $x_1, x_2\in\R^d$ such that $|x_1 - y| > 2|x_1 - x_2| $ (which implies $|x_1-y| \approx |x_2-y|$), we use the mean value theorem and the fact that $A$ is a Lipschitz function to get 
    \begin{align*}
    |K_A(x_1, y) - K_A(x_2,y)| 
    &\leq \Big| \frac{1}{|x_1 - y|^{d+1}} - \frac{1}{|x_2 - y|^{d+1}} \Big| \cdot |A(x_1) - A(y)|	\\
    & \qquad + \frac{1}{|x_2-y|^{d+1}} \big| A(x_1) - A(x_2) \big| \\
    &\lesssim \|\nabla A\|_{L^{\infty}(\R^d)} \frac{|x_1 - x_2|}{|x_1 - y|^{d+1}}.
    \end{align*}
Thus, the first inequality in (\ref{regularity condition}) holds. 
The second inequality follows similarly. 
This completes the proof.    
\end{proof}

\medskip
\subsection{Maximal Calder\'on commutator of Baj\v{s}anski-Coifman type} 
We next consider the Baj\v{s}anski--Coifman type operator.

    \begin{theorem}
    Let $l\geq 1$ be an integer.
    Suppose $\Omega\in L\log L(\sS^{d-1})$ satisfies (\ref{homogeneous function of degree zero}) and (\ref{cancellation condition of order l}).   
    Assume that $A\in C^{l-1}(\R^d)$ and that all its derivatives of order $l-1$ are globally Lipschitz.     
    Then, the maximal operator $T_{\Omega,A,l,*}$ defined as in (\ref{maximal Calderon commutator of B-C type}) is of weak type $(1,1)$.
    That is, for any $f\in L^1(\R^d)$ and $\lambda>0$, 
 	\[
 	\lambda \big| \{x\in\R^d: T_{\Omega, A, l,*} f(x)>\lambda \} \big| 
 	\lesssim C_{\Omega} \sum_{|\alpha| = l} \|A_{\alpha}\|_{L^{\infty}(\R^d)} \|f\|_{L^1(\R^d)}.
 	\]
    \end{theorem}

\begin{proof}
The case $l=1$ follows from the preceding theorem.
Suppose that $l\geq2$, and set
    \[
    K_{A,l}(x,y)=\frac{P_l(A,x,y)}{|x-y|^{d+l}},
    \]
with $P_l$ as defined as in (\ref{P_l(A,x,y)}).
We have known that $T_{\Omega,A,l,*}$ is bounded on $L^2(\R^d)$ by 
\cite{Bajvsanski1967On}.
Thus, using Theorem \ref{main theorem}, it is enough to show that the kernel $K_{A,l}$ satisfies the size condition (\ref{Size condition}) and the regularity condition (\ref{regularity condition}).

First, since \(A_\alpha \in L^\infty(\mathbb{R}^d)\) for every \(|\alpha| = l\), the Taylor expansion
    \[
    P_l(A,x,y) 
    = l \sum_{|\alpha|=l} \frac{(x-y)^\alpha}{\alpha!} \int_0^1 (1-s)^{l-1} A_\alpha(y + s(x-y)) \dif s,
    \]
gives
    \[
    |K(x,y)| 
    \lesssim \sum_{|\alpha|=l} \|A_\alpha\|_{L^{\infty}(\R^d)} \frac{1}{|x-y|^d}.
    \]
So (\ref{Size condition}) holds.

Next, choose \(|x_1 - y| > 2|x_1 - x_2|\). 
Then \(|x_1 - y| \approx |x_2 - y|\). 
By the Taylor expansion, we write    
    \begin{align*}
    P_l(A,x,y) 
    &= P_{l-1}(A,x,y) - \sum_{|\alpha|=l-1} \frac{A_\alpha(y)}{\alpha!} (x-y)^\alpha \\
    &= (l-1) \sum_{|\alpha|=l-1} \frac{(x-y)^\alpha}{\alpha!} \int_0^1 (1-s)^{l-2} \big(A_\alpha(y + s(x-y)) - A_\alpha(y)\big) \dif s.
    \end{align*}
Note that \(A_\alpha\) is a Lipschitz function  for each \(|\alpha|=l-1\). 
By the mean value theorem, it is easy to see that
    \[
    |K(x_1,y) - K(x_2,y)| \lesssim \sum_{|\alpha|=l} \|A_\alpha\|_{L^{\infty}(\R^d)} \frac{|x_1 - x_2|}{|x_1 - y|^{d+1}}.
    \]
The second inequality in (\ref{regularity condition}) follows similarly. 
Hence, (\ref{regularity condition}) holds for \(K(x,y)\). 
This completes the proof.      
\end{proof}

\medskip
\subsection{Maximal general singular integral of Muckenhoupt type} 
Our final application concerns the Muckenhoupt type operator.


    \begin{theorem}
	Suppose $\Omega\in L\log L(\sS^{d-1})$ satisfies (\ref{homogeneous function of degree zero}) and  (\ref{cancellation condition}).
    Then, the maximal operator $T_{\Omega,ir,*}$ defined as in (\ref{maximal general singular integral of Muckenhoupt type}) is of weak type $(1,1)$. 
    That is, for any $f\in L^1(\R^d)$ and $\lambda>0$, 
 	\[
 	\lambda \big| \{x\in\R^d: T_{\Omega, ir,*} f(x)>\lambda \} \big| 
 	\lesssim C_{\Omega} \|f\|_{L^1(\R^d)}.
 	\]
    \end{theorem}

\begin{proof}
Set
   \[
    K_r(x,y) = \frac{1}{|x-y|^{d+ir}}.
    \]
By Theorems \ref{main theorem} and the $L^2$ boundedness of $T_{\Omega, ir,*}$ (see \cite[Theorem~8]{Muckenhoupt1960On}), it suffices to verify the kernel $K_r(x,y)$ satisfies the size condition (\ref{Size condition}) and the regularity condition (\ref{regularity condition}).

Clearly, \(|K_r(x,y)| = \frac{1}{|x-y|^d}\) gives (\ref{Size condition}).

Suppose \(|x_1 - y| > 2|x_1 - x_2|\), then \(|x_1 - y| \approx |x_2 - y|\). 
By the mean value theorem, we have
    \[
    \begin{aligned}
    &|K_r(x_1,y) - K_r(x_2,y)| \\
    &\le \left| \frac{1}{|x_1-y|^d} - \frac{1}{|x_2-y|^d} \right| + \frac{1}{|x_2-y|^d} \left| e^{-ir \ln |x_1-y|} - e^{-ir \ln |x_2-y|} \right| \\
    &\lesssim \frac{|x_1 - x_2|}{|x_1-y|^{d+1}}.
    \end{aligned}
    \]
So the first inequality in (\ref{regularity condition}) holds.     
Similarly, the second inequality in (\ref{regularity condition}) follows the same way. 
Hence, the proof is complete. 
\end{proof}

\appendix
\section{Two lemmas} \label{Appendix}

We need two classical results.
The first one is the H\"ormander-Mihlin multiplier theorem (see \cite{Grafakos2014Classical}).
\begin{lemma}[H\"ormander-Mihlin multiplier theorem]  \label{bounded T_m}
Let $m$ be a bounded complex-valued function on $\R^d\setminus\{0\}$.
Assume that for all multi-indices $\alpha$ with $|\alpha|\le \lfloor d/2\rfloor+1$, the following derivative estimate holds:
    \begin{align*}
    |\partial^{\alpha}_{\xi} m(\xi) |
    \leq M |\xi|^{-|\alpha|}. 
    \end{align*}
Define the multiplier operator $T_m$ by $\mathcal{F}(T_m h)(\xi)=m(\xi)\mF(h)(\xi)$.
Then for every $1<p<\infty$, $T_m$ is of strong type $(p,p)$ for $1<p<\infty$ with operator norm at most $C_d\bigl(M+\|m\|_{L^\infty(\R^d)}\bigr)$.
\end{lemma}

The second lemma is the Rademacher-Menshov theorem (see \cite{Demeter2008Maximal}), which plays an important role in linearizing the maximal operator.  
\begin{lemma} [Rademacher-Menshov theorem] \label{Rademacher-Menshov}
Let $(X,\sigma)$ be a measure space and $\{h_j\}_{1\leq j\leq N}$ be a sequence of measurable functions satisfying the Bessel-type inequality:
for any finite sequence $\{\epsilon_j\}_{1\leq j\leq N}$ with each $\epsilon_j\in \{-1,1\}$,
    \begin{align*}
    \Big\| \sum_{1\leq j \leq N } \epsilon_j h_j \Big\|_{L^2(X)}
    \leq B.
    \end{align*}
Then, the following maximal inequality holds
    \begin{align*}
    \Big\| \sup_{0<M\leq N} \Big|\sum_{1\leq j \leq M } h_j \Big| \Big\|_{L^2(X)}
    \lesssim B \log(2+N).
    \end{align*}
\end{lemma}

\bibliographystyle{amsplain}
\bibliography{rf.bib}

\end{document}